\pdfoutput=1
\documentclass[12pt,oneside]{amsart}
\usepackage[style=alphabetic]{biblatex}
\ifdefined\HCode   

\fi
\usepackage{geometry}
\usepackage{graphicx}
\usepackage{amsmath}
\usepackage{comment}
\usepackage[textsize=scriptsize]{todonotes}
\usepackage{amsfonts}
\usepackage{setspace}
\usepackage{amssymb}
\usepackage{tikz-cd}
\usepackage{amsthm}
\usepackage{spectralsequences}
\usepackage{mathrsfs} 
\usepackage{tikz}
\usepackage{dsfont}
\usetikzlibrary{matrix}
\usepackage{adjustbox}
\usepackage{mathtools}
\usepackage{hyperref}
\usepackage{xcolor}
\usepackage{enumitem}
\usepackage{tensor}
\usepackage{csquotes}
\usepackage[english]{babel}

\usepackage{mymacros}
\usepackage{cleveref}

\hypersetup{
	colorlinks,
	linkcolor={red!50!black},
	citecolor={blue!50!black},
	urlcolor={blue!80!black}
}

\tikzset{
	rot90/.style={anchor=south, rotate=90, inner sep=.5mm}
}
\tikzset{
	rot45/.style={anchor=south, rotate=-45, inner sep=.5mm}
}

    \theoremstyle{plain}
\newtheorem{theorem}{Theorem}[section]

\newtheorem{lemma}[theorem]{Lemma}
\newtheorem{proposition}[theorem]{Proposition}

\newtheorem{question}[theorem]{Question}
\newtheorem{notation}[theorem]{Notation}
    \theoremstyle{definition}
\newtheorem{definition}[theorem]{Definition}
\theoremstyle{remark}
\newtheorem{remark}{Remark}[theorem]

\theoremstyle{plain}

\makeatletter
\def\@tocline#1#2#3#4#5#6#7{\relax
  \ifnum #1>\c@tocdepth 
  \else
    \par \addpenalty\@secpenalty\addvspace{#2}%
    \begingroup \hyphenpenalty\@M
    \@ifempty{#4}{%
      \@tempdima\csname r@tocindent\number#1\endcsname\relax
    }{%
      \@tempdima#4\relax
    }%
    \parindent\z@ \leftskip#3\relax \advance\leftskip\@tempdima\relax
    \rightskip\@pnumwidth plus4em \parfillskip-\@pnumwidth
    #5\leavevmode\hskip-\@tempdima
      \ifcase #1
       \or\or \hskip 1em \or \hskip 2em \else \hskip 3em \fi%
      #6\nobreak\relax
    \hfill\hbox to\@pnumwidth{\@tocpagenum{#7}}\par
    \nobreak
    \endgroup
  \fi}
\makeatother

 \newtheorem{innercustomthm}{Theorem}
\newenvironment{customthm}[1]
  {\renewcommand\theinnercustomthm{#1}\innercustomthm}
  {\endinnercustomthm}

\title{The Borel cohomology of free iterated loop spaces}

\author{Ishan Levy}
\address{Department of Mathematical Sciences, University of Copenhagen, Denmark}
\email{ishanlevy97@gmail.com}\author{Justin Wu}
\address{Department of Mathematics, UC Berkeley, Berkeley, CA, USA}
\email{justinjwu@berkeley.edu}
\begin{document}

\begin{abstract}
    We compute the $\SOr(n+1)$-equivariant mod $2$ Borel cohomology of the free iterated loop space $Z^{S^n}$ when $Z$ is a mod $2$ generalized Eilenberg Mac Lane space. When $n=1$, this recovers Bökstedt and Ottosen's computation for the free loop space. The highlight of our computation is a construction of cohomology classes using an $\mathrm{O}(n)$-equivariant evaluation map and a pushforward map.
\end{abstract}
\maketitle
\setcounter{tocdepth}{1}
\tableofcontents

\section{Introduction}\label{intro}
Let $Z$ be a mod $2$ generalized Eilenberg--Mac Lane space of finite type\footnote{See \Cref{remark:finitetype} for a discussion of the finite type hypothesis, and how it can be avoided.} (which we will call GEM), and let $H^*$ denote the mod $2$ cohomology functor.

In \cite{Bkstedt1999HomotopyOO,Ottosen2003}, Bökstedt and Ottosen gave a functorial computation of $H^*(Z^{S^1}_{h\SOr(2)})$ for such $Z$. To do this, they construct a certain endofunctor $\ell$ of the category of unstable algebras over the Steenrod algebra, and show that $\ell(H^*(Z))$ is naturally isomorphic to $H^*(Z^{S^1}_{h\SOr(2)})$ for $Z$ a $\GEM$, where $\GEM$ is the category of GEMs and homotopy classes of maps\footnote{In fact, they prove this more generally when $Z$ has polynomial cohomology.}.

The main goal of this paper is to extend their work to higher dimensional spheres. That is, we functorially compute an associated graded algebra of $H^*(Z^{S^n}_{h\SOr(n+1)})$ for $n\geq 1$, the $n=1$ case essentially recovering their $\ell$ functor.

For a finite unstable algebra $A$ over the Steenrod algebra, there is a division endofunctor $(-:A)$ of the category of unstable algebras over the Steenrod algebra that is left adjoint to $(-)\otimes A$, and has the property that there is a natural isomorphism $(H^*(Z):H^*(Y)) \cong H^*(Z^Y)$ for $Z$ a GEM. We use $L_n$ to denote the functor $(-:H^*(S^n))$, and we show that $L_n(A)$ comes equipped with a natural derivation of cohomological degree $-n$ which we denote $d$. The functor $L_n$ is very explicit, and we refer the reader to \Cref{omegapresentation} and \Cref{remark:free} for more about it.

\begin{customthm}{A}\label{thm:mainreplacethis}
	Let $Z\in \GEM$, and $n \geq 1$. There is a natural multiplicative filtration on $H^*(Z^{S^n}_{h\SOr(n+1)})$ such that the associated graded is the cohomology of the commutative differential graded algebra (or cdga) with underlying algebra $L_n(H^*(Z))\otimes H^*(\BSO(n+1))$ with differential determined by $x\otimes 1 \mapsto dx\otimes w_{n+1}$ and $1\otimes w_i \mapsto 0$.
\end{customthm}
Here $w_{i}$ is the $i^{th}$ Stiefel--Whitney class in $H^{i}(\BSO(n+1))$. In the case, $n=1$, there are no multiplicative extension problems in \Cref{thm:mainreplacethis}, and so the cohomology of the cdga recovers the functor $\ell$ of B\"okstedt and Ottosen. In the general case, we obtain partial results about the multiplicative extension problems, which are collected in \Cref{verificationsection}. 

Letting $q: Z^{S^n} \to Z^{S^n}_{h\SOr(n+1)}$ denote the map to the homotopy orbits, the proof of \Cref{thm:mainreplacethis} proceeds by analyzing the Serre spectral sequence of the Borel fibration

\begin{center}
	\begin{tikzcd}
		Z^{S^n}\ar[r,"q"] &Z^{S^n}_{h\SOr(n+1)} \ar[r] & BSO(n+1)
	\end{tikzcd}
\end{center}

This spectral sequence can also be obtained as the homotopy fixed point spectral sequence of the $\SOr(n+1)$-action on the cochains of $Z^{S^n}$, and so we refer to it as \textit{the homotopy fixed point spectral sequence}. The $E_2$-term of this homotopy fixed point spectral sequence is $L_n(H^*(Z))\otimes H^*(\BSO(n+1))$, and we show that there are no differentials until the $E_{n+1}$-page, where the differential is given exactly by the cdga differential in the theorem. 

The most subtle point is to show that $E_{n+2} = E_{\infty}$, so that the cohomology of the cdga agrees with the associated graded of the cohomology of $H^*(Z^{S^n}_{h\SOr(n+1)})$. The $E_{n+2}$-page in this situation has the pleasant property that its multiplicative generators lie on the edges of the spectral sequence. To show that these multiplicative generators are permanent cycles, we explicitly construct cohomology classes in $H^*(Z^{S^n}_{h\SOr(n+1)})$ which represent generating classes on the $E_{n+2}$-page, as explained in the theorem below.

\begin{customthm}{B}\label{theorembintro}
	Let $Z$ be a space, and $n \geq 1$. For $0 \leq i \leq n$, there are natural transformations $$\delta: H^j(Z)\to H^{j-n}(Z^{S^n}_{h\SOr(n+1)})$$ $$\phi_i: H^j(Z) \to H^{2j-i}(Z^{S^n}_{h\SOr(n+1)})$$ such that if $Z \in \GEM$, the composite $H^*(Z^{S^n}_{h\SOr(n+1)}) \xrightarrow{q^*}
	H^*(Z^{S^n}) \cong L_n(H^*(Z))$ induced from the map $Z^{S^n} \to Z^{S^n}_{h\SOr(n+1)}$ sends  $\delta(a)$ to $da$, $\phi_i(a)$ to $Sq_ia$ for $i<n$, and $\phi_n(a)$ to $Sq_na + ada$.
	
	Moreover, when $Z \in \GEM$, the images of $\phi_i,\delta$ along with $w_i$ generate $H^*(Z^{S^n}_{h\SOr(n+1)})$ as an algebra.
\end{customthm}

The existence of the operation $\delta$ gives a canonical lift of the derivation $d$ on $L_n(H^*(Z))$ to take values in $H^*(Z^{S^n}_{h\SOr(n+1)})$. The other operations $\phi_i$ should be thought of as lifts of the $E_{n+1}$-power operations on mod $2$ cohomology to take values in the Borel cohomology of the free loop space. Indeed, it follows from the above theorem that $\phi_n$ pulled back along the composite $Z \to Z^{S^n} \to Z^{S^n}_{h\SOr(n+1)}$ gives the Steenrod operation $Sq_n$.

The construction of the operations $\phi_n$ comes from the equivariant evaluation map $Z^{S^n} \to Z^{2}$ evaluating at antipodal points of the sphere. This map is equivariant for the $O(n)$ action, where $O(n)$ is the stabilizer of the antipodal points as a set.

This equivariant evaluation map induces a map
\[ H^*(Z^2_{hC_2}) \otimes H^*(\BSOr(n)) \cong H^*(N^{\Or(n)}_{\SOr(n)}(Z)_{hO(n)}) \to H^*(Z^{S^n}_{h\mathrm{O}(n)})\]

The latter cohomology $H^*(Z^{S^n}_{h\mathrm{O}(n)})$ turns out to be a free module of rank $n$ over $H^*(Z^{S^n}_{h\SOr(n+1)})$ with basis $t^i, 0 \leq i \leq n$, where $t$ is the class coming from $H^1(BC_2)$. We then define the classes $\phi_i$ to be the coefficients of the image of the total power operation class $P(x) \in H^*(Z^2_{hC_2})$ with respect to this basis. We implement the extraction of the coefficients with respect to this basis using a pushforward map in cohomology from the $\mathrm{O}(n)$-homotopy orbits to the $\SOr(n+1)$-homotopy orbits.

In addition to describing the cohomology of $Z^{S^n}_{h\SOr(n+1)}$ as the cohomology of a cdga, we give a presentation of the cohomology of this cdga, in generators and relations similar to Bökstedt and Ottosen's functor $\ell$. 

\begin{customthm}{C}\label{thmcintro}
  Let $Z \in \GEM$, and let $A$ be an unstable algebra over the Steenrod algebra. For $A = H^*(Z)$, the ring $\overline{\ell}_n(H^*(Z))$ (see \Cref{lbardefn} for an explicit presentation) is naturally isomorphic to an associated graded algebra of $H^*(Z^{S^n}_{h\SOr(n+1)})$.
\end{customthm}

In this paper, it is convenient to use the language of zeroth derived functors as ways of explaining the origin of many of the functors we use. Given a contravariant functor out of $(h\cS)^{\ft}$, the homotopy category of finite type spaces, its zeroth derived functor is obtained by restricting to the category of finite Eilenberg Mac Lane spaces $\GEM^{\fin}$, and then left Kan extending along the functor $H^*:\GEM^{\fin} \to \cK$ into the category of unstable algebras over the Steenrod algebra.

In particular \Cref{thmc} is equivalent to the fact that the functor $\overline{\ell}_n$ is the zeroth derived functor of the functor taking a space $Z$ to $\gr H^*(Z^{S^n}_{h\SOr(n+1)})$. We note that if we understood all the relations between the classes in $H^*(Z^{S^n}_{h\SOr(n+1)})$, it would be possible to refine \Cref{thmc} to compute the functor $\ell_n$, which is defined to be the zeroth derived functor of $H^*(Z^{S^n}_{h\SOr(n+1)})$.

\begin{question}\label{notassociatedgraded}
Can one present $\ell_n$, the zeroth derived functor of $Z \mapsto H^*(Z^{S^n}_{h\SOr(n+1)})$, using $\phi_i,\delta,w_i$ as generators?
\end{question}

We also wonder if there is an algebraic way to come up with a presentation for the functor $\ell_n$, similarly to how the zeroth derived functor of $H^*(-^{S^n})$ can be constructed purely algebraically as the left adjoint of $(-)\otimes H^*(S^n)$ in the category of unstable algebras over the Steenrod algebra.

The functors $\ell_n$ and $\overline{\ell}_n$ are like higher dimensional analogs of $\ell$, and can be thought of as a version of negative cyclic homology for the free iterated loop space that is constructed just using Steenrod algebraic data.

We turn to posing some further directions related to our work we think would be interesting to explore.

The zeroth derived functor $\ell_n$ is just the first in a family of (nonabelian) derived functors that approximate $H^*(Z^{S^n}_{h\SOr(n+1)})$. The higher derived functors are defined purely algebraically in relation to the zeroth one.

\begin{question}
Can the higher derived functors of $\ell_n$ be computed, at least in some nontrivial examples?
\end{question}

There is a Bousfield homology spectral sequence whose $E_2$ term is the derived functors of $\ell_n(H^*(Z))$ and which (under favorable conditions) converges to $H^*(Z^{S^n}_{h\SOr(n+1)})$. This was worked out in the case $n=1$ in \cite{bokstedt2004spectral} by Bökstedt and Ottosen, and they moreover found that for $Z=\CC\PP^n$, the spectral sequence degenerates at $E_2$, converging to $H^*((\CC\PP^n)^{S^1}_{hSO(2)})$.

Next, we speculate about the extent to which our results can be generalized.

\begin{question}
For which real $G$-representations $V$ can one compute $H^*(Z^{S(V)}_{hG})$, where $Z \in \GEM$?
\end{question}

In the above question, $S(V)$ are the unit vectors in the representation $G$. We suspect that our results should generalize in a straightforward manner to the standard representations of $\mathrm{U}(n)$ and $\mathrm{Sp}(n)$.

\begin{question}
Are there analogs of our computations at odd primes?
\end{question}

In \cite{Ottosen2003}, Ottosen constructs an analog of $\ell_1$ at odd primes, but it is unclear what the generalization is for larger $n$. Namely,
the construction of $\ell_1$ at odd primes fundamentally uses the subgroups $C_p \subset \SOr(2)$. At the prime $2$, the analog of these for larger $n$ was $\mathrm{O}(n) \subset \SOr(n+1)$, but there isn't an obvious odd primary analog.

\subsection*{Outline of the paper}

This paper is organized as follows:

In Section \ref{nonequivariant}, we study zeroth derived functors, and explain how Lannes' division functor gives an example. We give a presentation of division by $H^*(S^n)$.

In Section \ref{section4}, we explain some facts about $\SOr(n+1)$-spaces relevant to the main results of the paper.

In Section \ref{section5}, we construct classes of $H^*(Z^{S^n}_{h\SOr(n+1)})$ by producing the natural transformations $\delta,\phi_i$, and compute their images in $H^*(Z^{S^n})$.

In Section \ref{sectionmainspectralsequence}, we use the results of Section \ref{section5} to completely describe the homotopy fixed point spectral sequence computing $H^*(Z^{S^n}_{h\SOr(n+1)})$ for $Z \in \GEM$.

In \Cref{verificationsection}, we prove some relations that hold for the classes $\delta,\phi_i$.

In Section \ref{algebraicstuffsection}, we compute the zeroth derived functor of the functor taking a space $Z$ to the associated graded of $H^*(Z^{S^n}_{h\SOr(n+1)})$.

\subsection*{Notation}
\begin{enumerate}
    \item We use $h\cS$ to denote the homotopy category of spaces, and $h\cS^{\ft}$ to denote those which are finite type, ie have finitely many cells in each degree.
    \item We use $\GEM$ to denote the homotopy category of finite type mod $2$ generalized Eilenberg--Mac Lane spaces, and $\GEM^{\fin}$ to denote those which are finite (ie truncated).
    \item We use $H^*(-)$ to denote homology with mod $2$ coefficients.
    \item We use $F_n(-)$ to denote the functor $(-)^{S^n}_{hSO(n+1)}$.
    \item We use $\cU$ to denote the category of unstable modules over the Steenrod algebra as in Schwartz's book \cite[Definition 1.3.1]{lionelschwartzbook}, and use $\cK$ to denote the category of unstable algebras over the Steenrod algebra \cite[Definition 1.4.1]{lionelschwartzbook}.
    \item Given a space $X$ with a $G$-action, we use $X_0$ to denote $X$ with the trivial $G$-action, and use $\alpha:X_0\times G \to X$ to denote the $G$-equivariant action map.
\end{enumerate}
    
\subsection*{Acknowledgements}
We are grateful to Haynes Miller for introducing us to unstable modules over the Steenrod algebra, for suggesting this project, and for helpful conversations and comments. This work was done as part of the MIT PRIMES program, and we thank PRIMES for this research opportunity. We thank Piotr Pstragowski for a helpful conversation about derived functors. The first author was supported by the NSF Graduate Research Fellowship under Grant No. 1745302.

\section{Cohomology and derived functors}\label{nonequivariant}

In this section, we explain the construction and universal properties of the zeroth derived functor of a functor out of $h\cS^{op}$. We explain Lannes' division functors and how they give an example of a zeroth derived functor.

Let $\cU$ denote the symmetric monoidal category of unstable modules over the Steenrod algebra (see \cite{lionelschwartzbook}), so that $\cK$, the category of unstable algebras over the Steenrod algebra, is the category of commutative monoids in $\cU$ satisfying $Sq_0(x) = x^2$, where $Sq_i = Sq^{|x|-i}$.

\begin{lemma}\label{zeroderivedfunctorcharacterization}
    The functor $H^*:(\GEM^{\fin})^{op} \to \cK$ identifies the target with the cocompletion of the source with respect to sifted colimits (in $1$-categories).
    In particular, if $D$ is a $1$-category with sifted colimits, and $F:(\GEM^{\fin})^{op} \to D$ is a functor, then there is a unique sifted colimit preserving functor $\tilde{F}:\cK \to D$ extending $F$ along $H^*$.
\end{lemma}
\begin{proof}
It is easy to see that the category $\cK$ is the same as presheaves out of $(\GEM^{\fin})^{op}$ that preserve products. Indeed, such a presheaf is determined by its value on $K(\ZZ/2\ZZ,m)$, which corresponds to the $m^{\text{th}}$-graded piece of an unstable algebra, and the functoriality corresponds to the Steenrod operations, the multiplication, and the instability condition. The category of product preserving presheaves is the sifted cocompletion, finishing the proof.
\end{proof}

\begin{definition}\label{definition:zerothderivedfunctor}
    Given a functor $F:(h\cS^{\ft})^{op} \to D$, where $D$ has sifted colimits, we can restrict it to $(\GEM^{\fin})^{op}$, and then construct the \textit{zeroth derived functor}, which is the left Kan extension $\overline{F}:\cK \to D$ that exists by the above lemma. Note that there is a natural isomorphism between $\overline{F} \circ H^*$ and $F$ on $(\GEM^{\fin})^{op}$.
\end{definition}

Often, the zeroth derived functor also agrees with $F$ on the larger category $\GEM^{op}$. The lemma below gives a criterion for when this is true.

\begin{definition}\label{defn:postnikovconvergent}
    We say that a functor $F:(h\cS^{\ft})^{op} \to D$ is \textit{Postnikov convergent} if it sends the Postnikov tower of an object of $\GEM$ to a filtered colimit diagram.
\end{definition}

\begin{lemma}\label{lemma:nattransfinitetype}
If $F:(h\cS^{\ft})^{op} \to D$ is a Postnikov convergent functor, and $\overline{F}:\cK \to D$ its zeroth derived functor, then there is a unique natural transformation $\eta:\overline{F}\circ H^*\to F$ extending the natural isomorphism that exists on the subcategory $(\GEM^{\fin})^{op}$. Moreover $\eta$ is an isomorphism on $(\GEM)^{op}$.
\end{lemma}
\begin{proof}
Let $\FF_2\otimes\Sigma^{\infty}_+(-)$ denote the functor taking a space to its $\FF_2$-valued chains, viewed as a spectrum, and let $\Omega^{\infty}$ denote the $0$th space of a spectrum. Then given a space $X$, we can consider $RX = \Omega^{\infty}(\FF_2\otimes \Sigma^{\infty}_+X)$, which is the Eilenberg Mac Lane space whose homotopy groups are the homology of $X$. Since $R$ comes from an adjunction, it is part of a monad, and gives a cosimplicial resolution $R^{\bullet + 1}(X)$ of $X$ in terms of finite-type Eilenberg Mac Lane spaces. We can combine this with the Postnikov tower to get a tower of cosimplicial objects $\tau_{\leq n} R^{\bullet + 1}(X)$, where $\tau_{\leq n}$ denotes the $n$th Postnikov truncation.

The functor $H^*:(h\cS^{\ft})^{op} \to \cK$ sends this diagram to a colimit diagram, since it is Postnikov convergent and $R^{\bullet+1}(X)$ is  the canonical resolution of $H^*(X)$ by free unstable algebras. Since $\overline{F}$ preserves sifted colimits, it follows that $\overline{F}\circ H^*$ also sends this to a colimit diagram. It then follows that there is a unique natural transformation $\eta$ extending the isomorphism on $(\GEM^{\fin})^{op}$ coming from the assembly map of this colimit diagram for $F$, ie the unique map making the square below commute. 
\begin{center}
	\begin{tikzcd}
		 \colim(\overline{F}\circ H^*(\tau_{\leq n}R^{\bullet+1}(X)))\ar[r,"\cong"]\ar[d,"\cong"] & \overline{F}\circ H^*(X)\ar[d,"\eta_X"] \\
		 \colim(F(\tau_{\leq n}R^{\bullet+1}(X)))\ar[r] &F(X) 
	\end{tikzcd}
\end{center}

Since both the source and target of $\eta$ are Postnikov convergent, it follows that $\eta$ is an isomorphism on $\GEM^{op}$.
\end{proof}

\begin{remark}\label{remark:finitetype}
    The reason we work with finite type spaces is that free unstable algebras over the Steenrod algebra come from the cohomology of mod $2$ generalized Eilenberg--Mac Lane spaces only when they are finite type. This issue is not serious, and can be avoided by working with homology instead of cohomology, ie in a category of unstable coalgebras over the Steenrod algebra. Note also that for most functors of interest such as $H^*((-)^{S^n}_{hSO(n+1)})$ which we apply zeroth derived functors to send filtered colimits to limits, and are therefore determined by their value on objects of finite type.
\end{remark}

Now we recall Lannes' division functor, and how it gives an example of a zeroth derived functor. 

\begin{definition}\label{notation:division}
    Given $A \in \cK$ of finite type, the functor $(-)\otimes A$ has a left adjoint \cite[Proposition 3.8.2]{lionelschwartzbook}, which we denote $((-):A)$, and call \textit{division by $A$}.
\end{definition}

\begin{proposition}\label{lannesderivedfunctor}\label{formalnonequivariantisomorphism}
Let $Y\in h\cS^{\ft}$. There is a natural map $\theta_Y:(H^*(Z):H^*(Y)) \to H^*(Z^Y)$ that is an isomorphism when $Z \in (\GEM^{\fin})^{op}$ or $Z\in \GEM$ and $Y$ is a finite space. Moreover, $\theta_Y$
identifies $(H^*(-):H^*(Y))$ with the zeroth derived functor of $H^*(-^Y)$.
\end{proposition}
\begin{proof}
$(-:H^*(Y))$ preserves colimits, so is left Kan extended from $(\GEM^{\fin})^{op}$ by \Cref{zeroderivedfunctorcharacterization}. It is thus enough to produce the natural transformation $\theta_Y$ and check it is an isomorphism on objects of $\GEM$.

There is an evaluation map $\mathrm{ev}_Y:Z^Y\times Y \to Z$ at the level of spaces, which upon taking cohomology gives a map $H^*(Z) \to H^*(Z^Y)\otimes H^*(Y)$. We define $\theta_Y$ to be the map that is mate to this map via the adjunction defining the division functor.

When $Y$ is finite, then both the source and target are Postnikov convergent, so it is enough to see that this comparison map is an isomorphism when $Z \in (\GEM^{\fin})^{op}$. For this we note that the functor $Z \to Z^{Y}$ on $(\GEM^{\fin})^{op}$ is a coproduct preserving endofunctor, and so it suffices to check the result for $Z = K(\FF_2,n)$. But $H^*(Z^Y)$ corepresents the functor that is maps of graded vector spaces out of the dual of $H^*(Y)$ shifted up by $n$, and so does $(H^*(Z):H^*({Y}))$ by the universal property of the left adjoint and the fact that $H^*(Z)$ corepresents $H^n$.
\end{proof}

\begin{notation}\label{notation:thetan}
   We use $L_n(A)$ to denote $(A:H^*(S^n))$ as an endofunctor of $\cK$, and use $\theta_n: L_n (H^*(Z)) \to H^*(Z^{S^n})$ to denote the natural transformation $\theta_{S^n}$ from \Cref{lannesderivedfunctor}.
\end{notation}

The following proposition is well known, but we include a proof.
\begin{proposition}\label{omegapresentation}
As a commutative $A$-algebra, $L_n(A)$ is generated by generators $da$ in degree $|a|-n$ for each $a \in A$, along with the relations:
\begin{enumerate}
    \item $d(a+b) = da + db$
    \item $d(ab) = d(a)b + ad(b)$
    \item $(da)^2 = d(Sq_n a)$
    \item $d(Sq_i a)=0$ for all $n > i \geq 0$. \\ 
\end{enumerate}
The action of the Steenrod algebra is determined by $Sq^i(da) = d(Sq^ia)$ and the Cartan formula, and the universal map $A \to L_n(A)\otimes H^*(S^n)$ sends $a \mapsto a\otimes 1 + da\otimes y_n$ where $y_n$ is the nontrivial element of $H^n(S^n)$.
\begin{proof}
We use $\FF_2\{-\}_{\cK}$ to denote the functor taking a graded $\FF_2$-vector space $V$ to the free object on $V$ in $\cK$, ie the left adjoint of the forgetful functor into graded $\FF_2$-vector spaces. Let $N = H^*(S^n)$. It follows from \Cref{lannesderivedfunctor} that there is an isomorphism $(\FF_2\{V_*\}_{\cK}:N) \cong \FF_2\{V_*\otimes N^*\}_{\cK}$, since $\FF_2\{V_*\otimes N^*\}_{\cK}$ is the cohomology of the mapping space from $S^n$ to the $\GEM$ associated to $V_*$.

There is a unique graded basis $1,d$ for $N^*$, and so the generators in $\FF_2\{V_*\otimes N^*\}_{\cK}$ of the form $v\otimes 1$ we  call $v$, and the ones of the form $v \otimes d$ we call $dv$. Moreover, following the isomorphisms show that the coevaluation is given by $v\mapsto v\otimes1 +dv\otimes y_n$. More generally, given $a \in A \in \cK$, giving a class $a \in A_n$ is the same as a map $a:\FF_2\{\Sigma^n\FF_2\}_{\cK} \to A_n$, and we define $a$ and $da$ to be the images of the corresponding classes after applying $((-):N)$. Clearly relation $(1)$ holds.

Applying $Sq^i$ on the coevaluation map, we see $Sq^ia\mapsto Sq^ia\otimes1 +Sq^i(da)\otimes y_n$ so we must have $dSq^ia = Sq^ida$. From this it follows that relations $(3)$ and $(4)$ hold. Furthermore, applying the formula for the coevaluation on a product, we find $ac \mapsto (a\otimes 1 + da\otimes y_n)(c\otimes 1+dc\otimes y_n)$, yielding relation $(2)$.

Now an arbitrary algebra $A$ can be presented canonically as a pushout \begin{equation*}
\begin{tikzcd}
\FF_2\{A\otimes A\}_{\cK} \otimes \FF_2\{\bigoplus_{0}^\infty \Sigma^iA \}_{\cK}\arrow[d] \arrow[r] & \FF_2\{A\}_{\cK} \arrow[d] \\
0 \arrow[r]     & A
\end{tikzcd}
\end{equation*}
The vertical nonzero map is given by $[a] \mapsto a$ and the horizontal one sends $[a\otimes b]$ to $[a][b]-[ab]$, and $a \in \Sigma^iM$ to $Sq^i[a]-[Sq^ia]$.

Applying $((-):N)$, since $((-):N)$ preserves pushouts, we get a pushout square

\begin{equation*}
\begin{tikzcd}
\FF_2\{A \otimes A\otimes N^*\}_{\cK} \otimes \FF_2\{\bigoplus_{0}^\infty \Sigma^iA\otimes N^*\}_{\cK} \arrow[d] \arrow[r] & \FF_2\{A\otimes N^*\}_{\cK} \arrow[d] \\
0 \arrow[r]     & (A:N)
\end{tikzcd}
\end{equation*}
This gives a presentation for $(A:N)$ as the quotient of $\FF_2\{-\}_{\cK}(A\otimes N^*)$ by relation $(2)$, and the additional relation $Sq^i(da) = dSq^i(a)$.

From this we can extract a presentation for $A:N$ as an algebra. The fact that $Sq^ida = dSq^ia$ simplifies things: The fact that $Sq_ia = 0$ for $i<0$ and $Sq_0a = a^2$ reduces to relations $(3)$ and $(4)$. This presentation then reduces to the presentation of an algebra: $a$ and $da$ are generators for $a \in A$ with relations $(1)-(4)$.
\end{proof}
\end{proposition}

\begin{remark}\label{remark:free}
    We note from the proof of \Cref{omegapresentation} that the functor $(-:H^*(S^n))$ sends $\FF_2\{M\}_{\cK}$ to $\FF_2\{M\oplus \Omega^nM\}_{\cK}$, where $\Omega^n$ shifts the gradings of a grading vector space down by $n$, ie $(\Omega^nM)_i = M_{i+n}$.
\end{remark}

\section{\texorpdfstring{$\SOr(n+1)$}{} actions}\label{section4}
In this section, we collect relevant facts about spaces with $\SOr(n+1)$-actions, to set up our understanding of the $\SOr(n+1)$ action on the mapping space from $S^n$. Many of the results here are well known. Throughout this section, $X$ denotes a left or right $\SOr(n+1)$ space. 
\subsection*{Pushforward}\label{pushforwardsubsection}

An important construction we use later is that of the pushforward on cohomology. It makes sense in great generality, but we are interested in it for a fibration $f:E \to B$ whose fiber $F$ is a compact manifold of dimension $n$. We refer the reader to \cite{pushforwardref} and \cite{ando2018parametrized} for a reference on this construction. The pushforward is a natural map $f_*:H^*(E;H^n(F)) \to H^{*-n}(B)$ that can be defined in two equivalent ways, which we now describe.

Let $\nu$ be the normal bundle of the fibration $f:E \to B$. This is a virtual bundle of dimension $-n$ on $E$, whose fiber is the negative of the vectors tangent to the fiber. Then there is a Gysin map $\Sigma^{\infty}_+B \to E^{\nu}$, where $E^{\nu}$ is the Thom spectrum of $\nu$. The pushforward is then defined as the composite $H^*(E) \cong H^{*-n}(E^{\nu}) \to H^{*-n}(B)$, where the first map is the Thom isomorphism, and the second is the Gysin map.

If $F$ is connected, then the pushforward map can be identified with the map $H^*(E) \to H^*(B;H^n (F)) \cong H^{*-n}(B)$ where $H^*(E)$ is identified with the cohomology of $B$ with coefficients in the local system that is the cochains on the fiber. Indeed this follows from the axiomatic characterization of \cite[Theorem 3.1]{pushforwardref}. In this case, the $n^{th}$ row of the $E_{\infty}$ page of the Serre spectral sequence is a quotient of $H^*(E)$. The $E_{\infty}$ page injects into the $E_2$ page, so the composite $H^*(E) \to E_\infty^{*-n,n} \to E_2^{*-n,n}= H^{*-n}(B)$ gives the pushforward map.

The properties of the pushforward we need are:

\begin{proposition}

Let $f: E \to B$ be a fibration whose fiber $F$ is a compact manifold of dimension $n$. Then we have:
    \begin{enumerate}
    \item $f_*$ is a map of $H^*(B)$ modules (see \cite[page 5]{ando2018parametrized})
    \item $f_*$ is natural with respect to pullback fibrations.
    \item $g_*\circ f_* = (g\circ f)_*$.\footnote{This is because Gysin maps compose, and the Thom isomorphism is multiplicative.}
\end{enumerate}
\end{proposition}

We are primarily interested in the following situation. Suppose $G,H$ are compact Lie groups. If $X$ is a $G$-space and $H \subset G$ is a subgroup, then there is a fibration $X_{hH} \to X_{hG}$ with fiber $G/H$. The pushforward of this fibration is denoted $\tau_H^G$.

\subsection*{Coactions and comparisons}
The group $\mathrm{O}(n)$ sits inside $\SOr(n+1)$ as the stabilizer of any $1$-dimensional subspace of the standard representation of $\SOr(n+1)$. First we compare the cohomologies of the $\mathrm{O}(n)$ and $\SOr(n+1)$ quotients of an $\SOr(n+1)$-space. We recall that $H^*(\BOr(n))$ is a polynomial ring on the Stiefel--Whitney classes $w_i \in H^i(\BOr(n))$ for $1\leq i \leq n$, and $H^*(\BSOr(n))$ is polynomial ring on $w_i$ for $2 \leq i \leq n$.

\begin{lemma}\label{freemodule}
The map $\BOr(n) \to \BSOr(n+1)$ sends $w_i$ in $H^*(\BSO(n+1))$ to $w_i+w_{i-1}w_1$ in $H^*(\BOr(n))$. This realizes $H^*(\BOr(n))$ as a free module over $H^*(\BSOr(n+1))$. Two bases of this free module are given by $1,w_1,\dots, w_n$ and $w_1^i, 0 \leq i \leq n$. 

For any $SO(n+1)$-space $X$, the limit comparison map of $H^*(-)$ applied to $X_{h\mathrm{O}(n)} \cong \BOr(n)\times_{\BSOr(n+1)}X_{h\SOr(n+1)}$ induces an isomorphism

 \[H^*(\BOr(n))\otimes_{H^*(\BSOr(n+1))}H^*(X_{h\SOr(n+1)}) \xrightarrow{\cong} H^*(X_{h\mathrm{O}(n)})\]
\begin{proof}
The map $\BOr(n) \to \BSOr(n+1)$ is given by taking an $n$-plane bundle and adding a copy of its determinant bundle. By the Whitney sum formula, we get that $\sum_0^{n+1} w_i$ in $H^*(\BSOr(n+1))$ pulls back to $(\sum_0^{n} w_i)(1+w_1)$. Thus $w_i$ pulls back to $w_i+w_1w_{i-1}$, where we interpret $w_i$ to be $0$ when it is $0$ for the universal bundle. The fiber of the map is $\RR\PP^n$, where the map to $\BOr(n)$ is the orthogonal complement of the tautological bundle. Thus by the Whitney sum formula, its total Stiefel-Whitney class is $\frac{1}{t+1} = 1+t+\dots+t^n$. Thus $w_i$ get sent to a basis of $H^*(\RR\PP^n)$. Now considering the Serre spectral sequence of the fibration $\RR\PP^n \to \BOr(n) \to \BSOr(n+1)$, and comparing Poincare polynomials, there is no room for differentials, so the spectral sequence degenerates at $E_2$, and the $w_i$ generate $H^*(\BOr(n))$ freely. The basis $w_i$ is related to the classes $w_1^i$ via a triangular matrix with diagonal entries $1$, so both families work as a basis.

For the general case, consider the Eilenberg-Moore spectral sequence of the pullback square
\begin{center}
    \begin{tikzcd}
    X_{h\mathrm{O}(n)} \ar[r]\ar[d]\pullback & X_{h\SOr(n+1)}\ar[d]\\
    \BOr(n) \ar[r] & \BSOr(n+1)
    \end{tikzcd}
\end{center}

Since $H^*(\BOr(n))$ is a free $H^*(\BSOr(n+1))$-module, the $E_2$-term is concentrated in one line and is $H^*(\BOr(n))\otimes_{H^*\BSOr(n+1)}H^*(X_{h\SOr(n+1)})$. Thus the spectral sequence collapses at $E_2$, and gives the result we want.
\end{proof}
\end{lemma}

Next, we study the left action of $\SOr(n+1)$ on $S^n$. Recall that $H^* (\SOr(n))$ is isomorphic to $\otimes_{i \text{ odd}} \mathbb{F}_2 [x_i] / (x_i^{p_i})$ where $|x_i| = i$ and $p_i$ is the smallest power of $2$ such that $|x_i^{p_i}| \geq n$ \cite[\nopp 3D.2]{HatcherAT}. The following lemmas are well known.

\begin{lemma} \label{Snspectralseq}
The Serre spectral sequence of the fibration $\SOr(n) \to \SOr(n+1) \to S^n$ has no differentials. Moreover the nontrivial class $y_n \in H^n(S^n)$ pulls back to $x_i^{2^k}$ where $i2^k = n$.
\end{lemma}
\begin{proof}
The first statement follows from comparing Poincaré polynomials. Since the Serre spectral sequence degenerates, $y_n$ must generate the kernel of the map $H^*(\SOr(n+1)) \to H^*(\SOr(n))$, but this kernel has a unique generator which is the class indicated in the statement of the lemma.
\end{proof}

\begin{lemma} \label{BSOpathspace}
In the Serre spectral sequence of the fibration of the universal bundle $\SOr(n+1) \to \cdot \to \BSOr(n+1)$, the class $x_i^{2^k}$ transgresses to hit $w_{i2^k+1}$.
\begin{proof}
By induction on $n$ using the map $\BSOr(k) \to \BSOr(k+1)$ induced from the inclusion and comparing Serre spectral sequences for the path space fibrations, we only need to show this for $x_i^{2^k}$ where $2^k=n$. This class pulls back to $0$ in $H^*(\SOr(n))$, so by comparing the path space fibrations, it cannot have any differential other than a transgression. Via the transgression it must hit the class $w_{n+1}$ as it is the only class left on the $E_{n+1}^{n+1,0}$ and the $E_{\infty}$ page has $\FF_2$ in bidegree $0,0$.

However, by induction this is the only class in dimension $n$ whose differential has not been computed, and $w_{n+1}$ is the only class in dimension $n+1$ that has not been hit by a differential.
\end{proof}
\end{lemma}

\begin{lemma}\label{Sncoaction}
 The coaction of $H^*(\SOr(n+1))$ on $H^*(S^n)$ is given by $y_n \mapsto y_n\otimes 1 +1\otimes x_i^{2^k}$ where $i2^k = n$.
 \begin{proof}
 Since the action map $\SOr(n+1)\times S^n \to S^n$ is unital and $H^*(S^n)$ is 2-dimensional, $y_n$ must be sent to $y_n\otimes 1 + 1\otimes c$ for some class $c$. The class $c$ is then the image of $y_n$ in the composite $\SOr(n+1) \to \SOr(n+1)\times S^n \to S^n$. But the composite map is the fibration in Lemma \ref{Snspectralseq}, so we are done by that lemma.
 \end{proof}
\end{lemma}

Here is another related spectral sequence we use later. It can also be used to give an alternate proof that $y_n$ pulls back to $x^{2^k}$ in Lemma \ref{Snspectralseq}.

\begin{lemma}\label{borelsnspectralsequence}
The Serre spectral sequence of the fibration $\SOr(n) \to S^{n-1} \to \BSOr(n-1)$ has differentials sending $x_i^{2^k}$ to $w_{i2^k+1}$ for $i<n$.

\begin{proof}
This follows by considering the comparison map from this spectral sequence to the one in Lemma \ref{BSOpathspace}.
\end{proof}
\end{lemma}

\subsection*{The free iterated loop space}
We now compute the coaction and evaluation map relevant to the mapping space $Z^{S^n}$ where $Z$ is a space. Recall from \Cref{notation:thetan} that there is a natural transformation $\theta_n: L_n(H^*(Z)) \to H^*(Z^{S^n})$, and so given a class $a$ in $H^*(Z)$, there are corresponding classes $\theta_n(a),\theta_n(da)$ in $H^*(Z^{S^n})$.

The following lemma follows from the definition of $\theta_n$:
\begin{lemma}\label{coevaluationcohomology}
The coevaluation map $H^*(Z)\to H^*(Z^{S^n} \times S^n)$ on cohomology sends the class $a$ to $\theta_n(a)\otimes 1 + \theta_n(da) \otimes y_n$.
\end{lemma}

\begin{proposition}\label{coaction}
The coaction $H^*(Z^{S^n}) \to H^*(Z^{S^n})\otimes H^*(\SOr(n+1))$ sends $\theta_n(da) \mapsto \theta_n(da)\otimes 1, \theta_n(a) \mapsto \theta_n(a)\otimes 1 +  \theta_n(da)\otimes x_i^{2^k}$ where $i2^k=n$.

\begin{proof}
Consider the commutative diagram
\begin{center}
    \begin{tikzcd}
Z^{S^n}\times S^n \arrow[r, "ev"]                                           & Z                                 \\
Z^{S^n}\times \SOr(n+1)\times S^n \arrow[r, "\alpha_R"] \arrow[u, "\alpha_L"] & Z^{S^n}\times S^n \arrow[u, "ev"]
\end{tikzcd}
\end{center}
Where $\alpha_L$ is the left action on $S^n$ and $\alpha_R$ is the right action on $Z^{S^n}$. We can pull back the class $a \in H^*(Z)$ to $H^*(Z^{S^n}\times \SOr(n+1)\times S^n)$ along $ev$ and $\alpha_L$, which by Lemma \ref{coevaluationcohomology} and Lemma \ref{Sncoaction} is given by $a \mapsto \theta_n(a)\otimes 1 + \theta_n(da)\otimes y_n \mapsto \theta_n(a)\otimes 1\otimes 1 + \theta_n(da)\otimes y_n\otimes 1 + \theta_n(da)\otimes 1 \otimes x_i^{2^k}$. Thus by commutativity, this is where $\alpha_R^*$ sends the class $\theta_n(a)\otimes 1 + \theta_n(da)\otimes y_n$, giving the result.
\end{proof}
\end{proposition}

Next we end with some computations that are needed later. $\SOr(n)^2_{h\SOr(n-1)}$ denotes the quotient by the diagonal action.

\begin{lemma}\label{sonsquarecohomology}
The Serre spectral sequence of $\SOr(n) \to \SOr(n)^2_{h\SOr(n-1)} \to \SOr(n)_{h\SOr(n-1)}\cong S^{n-1}$ where the second map is given by projection onto one factor degenerates at $E_2$, as does the Eilenberg Moore sequence sequence for the product of two copies of $\SOr(n)/\SOr(n-1)$ over $\BSOr({n-1})$. Moreover, the map $\SOr(n)^2 \to \SOr(n)^2_{h\SOr(n-1)}$ is injective on cohomology.

\begin{proof}
For the first and last statement, one considers the map of spectral sequences coming from the following map of fibrations:
\begin{center}
\begin{tikzcd}
\SOr(n) \ar[d,equal] \arrow[r] & \SOr(n)^2_{h\SOr(n-1)} \arrow[r] & S^{n-1} \\
\SOr(n) \arrow[r] & \SOr(n)^2 \arrow[u] \arrow[r]    & \SOr(n)\ar[u].
\end{tikzcd}
\end{center}
Then one uses the injectivity of the map on cohomology and the lack of differentials in the target to conclude.
For the statement about the Eilenberg Moore spectral sequence, since the map $S^{n-1} \to \BSOr(n-1)$ is trivial on cohomology by Lemma \ref{BSOpathspace}, the $E_2$ term is $H^*(S^{n-1}) \otimes H^*(S^{n-1})\otimes \Tor^{*,*}(H^*(\BSOr(n-1)))$, which cannot have any differentials, as it is the same dimension as the cohomology of $\SOr(n)^2_{h\SOr(n-1)}$.
\end{proof}
\end{lemma}

\section{Equivariant evaluations}\label{section5}

The goal of this section is to construct the natural transformations $\phi_i,\delta$ that we use to prove \Cref{theorembintro}.

\subsection*{Construction and identification of the classes}

Here for an arbitrary space $Z$, we study classes in the cohomology of $Z^{S^n}_{h\SOr(n+1)}$. From the discussion in the introduction, when $Z \in \GEM$, the cohomology should be generated by the classes $w_i$ coming from $\BSOr(n+1)$, along with classes that under the map $q: Z^{S^n} \to Z^{S^n}_{h\SOr(n+1)}$ pullback to $\theta_n(\Sq^i a), i<n, \theta_n(\Sq^n a+ada), \theta_n(da)$ where $a$ is an arbitrary cohomology class in $H^*(Z)$. Our goal is to construct classes $\phi_i(a), \delta(a)$ in $Z^{S^n}_{h\SOr(n+1)}$ that pull back to these (Proposition \ref{propF}). Our construction uses the \textit{equivariant evaluation} maps. 

Before introducing the equivariant evaluation, we recall the following functor:
\begin{notation}\label{norm}
We use $N_{\SOr(n)}^{\Or(n)}(-)$ to denote the Norm/coinduction functor that is the right adjoint of the forgetful functor from spaces with an $\Or(n)$-action to spaces with an $\SOr(n)$-action.

If $Y$ is a space with an $\mathrm{O}(n)$-action, we use $\Delta_{O(n)}:Y \to N_{\SOr(n)}^{\Or(n)}(Y)$ to denote the unit of the adjunction.
\end{notation}

\begin{remark}
    Choose an element $x \in \mathrm{O}(n)$ of order $2$. Then $N_{\SOr(n)}^{\Or(n)}(Y)$ is explicitly given by the space $Y^2$ with $\mathrm{O}(n)$-action so that the $\SOr(n)$-action of an element $z$ sends $(a,b)$ to $(za,xzx^{-1}b)$ and $x$ sends $(a,b)$ to $(b,a)$. Moreover, $\Delta_{O(n)}(a) = (a,xa)$ in this description.
\end{remark}

\begin{notation}\label{notation:eqevaluation}
Let $ev_0$ be the map $Z^{S^n} \to Z_0$ given by evaluating at a basepoint of $S^n$, where $Z_0$ denotes $Z$ with a trivial action. Since $\SOr(n) \subset \SOr(n+1)$ is the stabilizer of a point in $S^n$, this map is $\SOr(n)$-equivariant.

Let $ev_1: Z^{S^n} \to N^{\Or(n)}_{\SOr(n)}(Z)$ be the $\Or(n)$-equivariant map corresponding to $ev_0$ under the adjunction of \Cref{norm}. Note that the $\Or(n)$ action on $N^{\Or(n)}_{\SOr(n)}(Z)$ factors through $C_2 = \Or(n)/\SOr(n)$.

The notation $ev_0,ev_1$ is used both to refer to these maps as well as the induced maps on homotopy quotients
$$ev_0: Z^{S^n}_{h\SOr(n)} \to Z_{h\SOr(n)}\cong Z\times \BSOr(n)$$ $$ev_1:Z^{S^n}_{hO(n)} \to N^{\Or(n)}_{\SOr(n)}(Z)_{hO(n)}$$
\end{notation}

Let $\pi$ denote the projection map $N^{\Or(n)}_{\SOr(n)}(Z)_{hO(n)}=(Z^2\times \BSOr(n))_{hC_2} \to Z^2_{hC_2}$

Recall (see for example \cite{nishida}) that for $a \in H^*(Z)$, there is a total power operation class $P(a)$ in $H^*(Z^2_{hC_2})$ with the property that its pullback to $Z^2$ is $a\otimes a$, and its pullback to the fixed points $Z\times \RR\PP^\infty$ is $\sum_i t^i\Sq_i(a)$. Let $t \in H^*(Z^2_{hC_2})$ denote the pullback of the generator $t \in H^* (BC_2)= H^*(\RR\PP^\infty)$.
\begin{definition}\label{deltaphidefn}

We define the natural map $$\delta: H^* (Z) \to H^{*-n}(Z^{S^n}_{h\SOr(n+1)})$$ to be the composite 
\[H^* (Z) \xrightarrow[]{\pi_1^*} H^*(Z\times \BSOr(n))= H^*(Z_{h\SOr(n)}) \xrightarrow[]{ev_0^*} H^*(Z^{S^n}_{h\SOr(n)}) \xrightarrow{\tau_{\SOr(n)}^{\SOr(n+1)}} H^{*-n}(Z^{S^n}_{h\SOr(n+1)}).\]

For $0\leq i \leq n$ we define the natural map $$\phi_i: H^j(Z) \to H^{2j-i}(Z^{S^n}_{h\SOr(n+1)})$$ by taking $\phi_i (a)$ to be the image of $t^{n-i} P(a)$ in the composite
\[H^*(Z^2_{hC_2}) \xrightarrow{\pi^*} H^*(N^{\Or(n)}_{\SOr(n)}(Z)_{hO(n)}) \xrightarrow{ev_1^*} H^*(Z^{S^n}_{h\mathrm{O}(n)}) \xrightarrow{\tau_{\mathrm{O}(n)}^{\SOr(n+1)}} H^*(Z^{S^n}_{h\SOr(n+1)}).\]
\end{definition}

Our goal is to prove the following result, which is a substantial part of \Cref{theorembintro}:

\begin{proposition}\label{propF}
For the map $q^*: H^*(Z^{S^n}_{h\SOr(n+1)}) \to 
H^*(Z^{S^n})$, we have $q^*(\delta(a))= \theta_n(da)$, $q^*(\phi_i(a)) = \theta_n(Sq_ia)$ for $0\leq i<n$, and $q^*(\phi_n(a)) = \theta_n(Sq_na + ada)$.
\end{proposition}

We begin by showing that $\delta(a)$ has the desired image in $H^*(Z^{S^n})$ in the Lemma below, which is the first part of \Cref{propF}.

\begin{lemma}\label{deltaimage}
The image of $\delta(a)$ via the map $q^*:H^*(Z^{S^n}_{h\SOr(n+1)}) \to H^*(Z^{S^n})$ is the class $\theta_n(da)$.

\begin{proof}
Let $\alpha: Z_0^{S^n}\times \SOr(n+1) \to Z^{S^n}$ be the $\SOr(n+1)$-equivariant map coming from the $\SOr(n+1)$-action on $Z^{S^n}$.
The key observation is that the homotopy quotient of $\alpha$ by $\SOr(n+1)$ is exactly the map $q:Z^{S^n} \to Z^{S^n}_{h\SOr(n+1)}$. Then we can consider the commutative diagram
\begin{center}
\begin{tikzcd}
H^*(Z_{h\SOr(n)}) \arrow[d, "i^*"] \arrow[r, "ev_0^*"] & H^*(Z^{S^n}_{h\SOr(n)}) \arrow[d, "\alpha^*"] \arrow[r, "\tau_{\SOr(n)}^{\SOr(n+1)}"] & H^*(Z^{S^n}_{h\SOr(n+1)}) \arrow[d, "q^*"] \\
H^*(Z) \arrow[r, "ev^*"]                             & H^*(Z_0^{S^n}\times S^n) \arrow[r, "\tau_{\SOr(n)}^{\SOr(n+1)}"]                    & H^*(Z_0^{S^n}).                               
\end{tikzcd}
\end{center}
By the definition of $\delta(a)$, we are trying to understand the image of $a$ from the top left corner to the bottom right corner. But following the lower part of the diagram, and using Lemma \ref{coevaluationcohomology} this gives $a \mapsto \theta_n(a)\otimes 1 + \theta_n(da)\otimes y_n \mapsto \theta_n(da)$.
\end{proof}
\end{lemma}

\begin{definition}\label{swapconstruction}
    Let $Y$ be a space with $\SOr(n+1)$-action. We define $f_Y$ to be the composite

    $$Y_0\times \SOr(n+1) \xrightarrow{\Delta_{O(n)}} N_{\SOr(n)}^{\Or(n)}(Y_0\times \SOr(n+1)) \xrightarrow{N_{\SOr(n)}^{\Or(n)}(\alpha)}N_{\SOr(n)}^{\Or(n)}(Y)$$

    where $\alpha$ is the $\SOr(n+1)$-action map. This is equivalently the map corresponding under the adjunction defining $N^{\Or(n)}_{\SOr(n)}$ to the $SO(n)$-equivariant action map $Y_0\times \SOr(n+1) \to Y$.
\end{definition}

To prove \Cref{propF}, it remains to understand the image of the elements $\phi_i(a)$ under $q^*$. One could try using the exact same strategy as in \Cref{deltaimage} to prove this, but one would have to understand the map $$Z^{S^n}\times \RR\PP^{n} \cong (Z_0^{S^n}\times S^n)_{hC_2} \to (Z^2)_{hC_2}$$ at the level of cohomology, which is $C_2$-equivariant analog of \Cref{coevaluationcohomology}. We don't see any direct way of doing this, so we approach understanding this via considering the diagram below, where $Z'$ denotes $Z^{S^n}_0\times \SOr(n+1)$, and we have identified $\SOr(n+1)/\mathrm{O}(n)=\RR\PP^n$.

\begin{center}
\begin{equation}\label{phidiagram}
\begin{tikzcd}[column sep=5pt]
H^*(Z^2_{hC_2})\arrow[r,"\pi^*"]&H^*(N_{\SOr(n)}^{\Or(n)}(Z)_{h\mathrm{O}(n)}) \arrow[r,"ev_1^*"] \arrow[d,"N_{\SOr(n)}^{\Or(n)}(ev_0)^*"]              &[15pt] H^*(Z^{S^n}_{h\mathrm{O}(n)}) \arrow[r,"\tau_{\mathrm{O}(n)}^{\SOr(n+1)}"] \arrow[d,"\alpha^*"]              &[10pt] H^*(Z^{S^n}_{h\SOr(n+1)}) \arrow[d,"q^*"] \\&
H^*(N_{\SOr(n)}^{\Or(n)}(Z^{S^n})_{h\mathrm{O}(n)}) \arrow[r,"f_{Z^{S^n}}^*"] \arrow[d,"N_{\SOr(n)}^{\Or(n)}(\alpha)^*"]        & H^*(Z^{S^n}_0\times \RR\PP^n) \arrow[r,"\tau_{\mathrm{O}(n)}^{\SOr(n+1)}"] \arrow[d,hook,"(\alpha \times 1)^*"] & H^*(Z^{S^n}_0)                    \\&
H^*(N_{\SOr(n)}^{\Or(n)}(Z')_{h\mathrm{O}(n)}) \arrow[r,"f_{Z'}^*"] & H^*(Z'_0\times \RR\PP^n)&                                
\end{tikzcd}
\end{equation}
\end{center}

In the the diagram above, we would like to compute the image of the class $\phi_i(a)$ along the map $H^*(Z^{S^n}_{h\SOr(n+1)})\to H^*(Z^{S^n})$. Because this class was constructed using $\tau_{\mathrm{O}(n)}^{\SOr(n+1)}$, this is image of the class $t^{n-i}P(a)$ in the diagram above from $H^*(N_{\SOr(n)}^{\Or(n)} (Z)_{h\mathrm{O}(n)})$ to $H^*(Z_0^{S^n})$. Identifying $H^*(Z^{S^n}_0\times \RR\PP^n)$ with $H^*(Z^{S^n})[t]/t^{n+1}$, the pushforward from $H^*(Z_0^{S^n}\times \RR\PP^n)$ to $H^*(Z_0^{S^n})$ extracts the coefficient of $t^n$ by the description of pushforwards in Section \ref{pushforwardsubsection} via the Serre spectral sequence. Thus we focus our efforts on understanding the map $$H^*(N_{\SOr(n)}^{\Or(n)}(Z)_{h\mathrm{O}(n)})\to H^*(Z_0^{S^n}\times \RR\PP^n)$$

The arrow indicated $1\times \alpha^*$ in the above diagram is an injection, so it suffices to understand the composite map $$H^*(N_{\SOr(n)}^{\Or(n)}(Z)_{h\mathrm{O}(n)})\to  H^*(Z_0^{S^n}\times \RR\PP^n) \to H^*(Z'_0\times \RR\PP^n)$$.

To do this, we need to understand the cohomology of $N_{\SOr(n)}^{\Or(n)}(Z')_{h\Or(n)} $. The map $Z' \to Z_0^{S^n}$ given by projection to the first factor induces a map $$ N_{\SOr(n)}^{\Or(n)}(Z')_{h\Or(n)} \to N_{\SOr(n)}^{\Or(n)}(Z_0^{S^n})_{h\mathrm{O}(n)},$$ and the pullback of a class $P(a)$ along this map is given the same name. Similarly, the projection to the second factor $Z' \to \SOr(n+1)$ induces a map $$N_{\SOr(n)}^{\Or(n)}(Z')_{h\Or(n)}  \to N_{\SOr(n)}^{\Or(n)}(SO(n+1))_{h\mathrm{O}(n)}.$$ We construct a class $P(y_n) \in H^*(N_{\SOr(n)}^{\Or(n)}\SOr(n+1)_{h\mathrm{O}(n)})$ whose pullback to $H^* (N_{\SOr(n)}^{\Or(n)}(Z')_{h\mathrm{O}(n)})$ is given the same name. 

\begin{lemma}\label{SOnsqOnquotient}
There is a unique class $P(y_n) \in H^*(N_{\SOr(n)}^{\Or(n)}(\SOr(n+1))_{h\mathrm{O}(n)})$ that pulls back to the class $(\pi_1^*y)(\pi_2^*y)$ in $H^*(N_{\SOr(n)}^{\Or(n)}(\SOr(n+1))_{h\SOr(n)})$, and to $0$ in $H^*(N_{\SOr(n)}^{\Or(n)}(\SOr(n))_{h\mathrm{O}(n)})$. Here $\pi_i$ is the map $(N_{\SOr(n)}^{\Or(n)}\SOr(n+1))_{h\SOr(n)} \to \SOr(n+1)_{h\SOr(n)} = S^n$ given by projection onto the $i$th factor, and $y_n\in H^n(S^n)$ is the nontrivial class.
\begin{proof}
Consider the following map of fibrations:
\begin{center}
\begin{tikzcd}
N_{\SOr(n)}^{\Or(n)}\SOr(n+1) \arrow[r] \arrow[d] & N_{\SOr(n)}^{\Or(n)}\SOr(n+1)_{hC_2} \arrow[r] \arrow[d,"\pi"] & BC_2 \arrow[d] \\
N_{\SOr(n)}^{\Or(n)}\SOr(n+1)_{h\SOr(n)} \arrow[r]  & N_{\SOr(n)}^{\Or(n)}\SOr(n+1)_{h\mathrm{O}(n)} \arrow[r]          & BC_2          
\end{tikzcd}
\end{center}
where the middle map is defined in \Cref{notation:eqevaluation} for $Z = SO(n+1)$. The left and right vertical maps are injective on cohomology (see \Cref{sonsquarecohomology}), so since all the classes of the form $a\otimes a$ of $H^*(\SOr(n+1)^2)\otimes H^0(BC_2)$ are permanent cycles in the Serre spectral sequence of the fibration (because of $P(a)$), the same is true of classes in $H^*(\SOr(n+1)^2)$ that pullback to ones of the form $a\otimes a$, for example $y_n\otimes y_n$. By looking at the relative Serre spectral sequences with respect to $\SOr(n) \subset \SOr(n+1)$, this class is seen to be unique as described.
\end{proof}
\end{lemma}

We now may finish the proof of \Cref{propF}.

\begin{proof}[Proof of \Cref{propF}]
The case of $\delta(a)$ was treated in Lemma \ref{deltaimage}, so we focus on computing the image of $\phi_i(a)$ in $H^*(Z^{S^n})$.

We first claim that the class $P(a) \in H^*(Z^2_{hC_2})$ is sent to the sum of three terms $P(\theta_n(a))+ \tau_{\SOr(n)}^{\mathrm{O}(n)} \theta_n(a)\otimes \theta_n(da) \otimes 1 \otimes y_n + P(\theta_n(da))P(y_n)$ in $H^*(N_{\SOr(n)}^{\Or(n)}(Z')_{h\mathrm{O}(n)})$.

To do this, we observe that the map $N_{\SOr(n)}^{\Or(n)}Z'_{h\mathrm{O}(n)} \to Z^2_{hC_2}$ factors through the map $(Z^{S^n}\times S^n)^2_{hC_2} \to Z^2_{hC_2}$ given by evaluation on each factor. Via the latter map, by Lemma \ref{coevaluationcohomology} $P(a)$ pulls back to $P(\theta_n(a)\otimes 1 + \theta_n(da)\otimes y_n)$. Since $P(cd) = P(c)P(d)$ and $P(c+d) = P(c) + P(d) + \tau_{*}^{C_2}(c\otimes d)$\footnote{These properties of $P$ are well known and easily verified using that the map $(1)$ in \cite[page 1]{nishida} is injective.}, we have

$P(\theta_n(a)\otimes 1+\theta_n(da)\otimes y_n)$ is equal to $P(\theta_n(a)) + P(\theta_n(da))P(y_n) + \tau_{*}^{C_2} (\theta_n(a)\otimes \theta_n(da) \otimes 1 \otimes y_n)$. Pulling this back to $N_{\SOr(n)}^{\Or(n)}(Z')_{h\mathrm{O}(n)}$ yields the claim.

Next we study each of the three terms. We can compute the image of $P(\theta_n(a))$ in $Z_0^{S^n}\times \RR\PP^n$ by examining the commutative square:
\begin{center}
    \begin{tikzcd}
Z'\times\RR\PP^n \arrow[r] \arrow[d] & N_{\SOr(n)}^{\Or(n)}(Z')_{h\mathrm{O}(n)} \arrow[d]          \\
Z_0^{S^n}\times \RR\PP^n \arrow[r]   & (N_{\SOr(n)}^{\Or(n)}(Z_0^{S^n})\times \BSOr(n))_{hC_2}
\end{tikzcd}
\end{center}
$P(\theta_n(a))$ is pulled back from the class with the same name in $H^*(N_{\SOr(n)}^{\Or(n)}(Z_0^{S^n})_{hC_2})$. That class pulls back to $\sum_i \Sq_i \theta_n(a) t^i$ in $H^*(Z_0^{S^n}\times \RR\PP^n)$ essentially by definition of the Steenrod operations.

For $\tau_{\SOr(n)}^{\mathrm{O}(n)} \theta_n(a)\otimes\theta_n(da) \otimes 1 \otimes y_n$, by the naturality of the pushforward, its image in $H^*(Z^{S^n}\times \RR\PP^n)$ can be obtained by pulling back the class $\theta_n(a)\otimes \theta_n(da) \otimes 1 \otimes y_n \in H^*((Z^{S^n}\times S^n)^2)$ along the diagonal to $Z^{S^n}\times S^n$, and then pushing forward to $Z^{S^n}\times \RR\PP^n$. Pulling back yields $\theta_n(ada)\otimes y_n$, and the pushing forward is a product with the pushforward map $S^n \to \RR\PP^n$, giving the class $\theta_n(ada)\otimes t^n$.

Finally, we examine the class $P(y_n)$ using the fact that it is pulled back from $N_{\SOr(n)}^{\Or(n)}(\SOr(n+1))_{h\mathrm{O}(n)}$ and the commutative diagram
\begin{center}
    \begin{tikzcd}
Z'\times\RR\PP^n \arrow[r] \arrow[d] & N_{\SOr(n)}^{\Or(n)}(Z')_{h\mathrm{O}(n)} \arrow[d] \\
\SOr(n+1)\times \RR\PP^n \arrow[r]     & N_{\SOr(n)}^{\Or(n)}(\SOr(n+1))_{h\mathrm{O}(n)}     
\end{tikzcd}
\end{center}
Then, for degree reasons, the class $P(y_n)$ pulls back along the bottom map to $\sum c_i \otimes t^i$, where $c_i$ are classes of degree $i>1$.

Putting it all together, the class $P(\theta_n(a)) \in H^*(Z^2_{hC_2)})$ is sent to $\sum c'_i \otimes t^i + \sum_i\theta_n(Sq_ia)\otimes t^i + \theta_n(ada)\otimes t^n$ in $H^*(Z'\times \RR\PP^n)$ where $c'_i$ are classes of positive degree. Because in the commutative diagram  (\ref{phidiagram}) above, the map indicated by $\hookrightarrow$ is injective, this means that $P(\theta_n(a))$ is sent to $\sum_0^{n}\theta_n(Sq_ia) t^i + \theta_n(ada)t^n$ in $H^*(Z^{S^n})\times \RR\PP^n$. Since the pushforward map $\tau_{\mathrm{O}(n)}^{\SOr(n+1)}:H^*(Z^{S^n}_0 \times \mathbb{RP}^n) \to H^*(Z^{S^n}_0)$ extracts the power of $t^n$, $t^i P(a)$ (and thus $\phi_i(a)$) is sent to $\theta_n(Sq_{n-i}a)$ for $i\neq 0$ and $\theta_n(Sq_na-ada)$ for $i=0$.
\end{proof}

\section{The homotopy fixed point spectral sequence}\label{sectionmainspectralsequence}

Recall that there is a homotopy fixed point spectral sequence computing $H^*(Z^{S^n}_{h\SOr(n+1)})$ with $E_2$-term $H^*(Z^{S^n})\otimes H^*(\BSO(n+1))$. This homotopy fixed point spectral sequence can also be interpreted as the Serre spectral sequence associated to the fibration $Z^{S^n} \to Z^{S^n}_{h\SOr(n+1)} \to \BSO(n+1)$. The goal of this section is to completely compute this spectral sequence when $Z \in \GEM$, thereby determining the Borel cohomology of the free iterated loop space of such a space, proving \Cref{thm:mainreplacethis}. A key step is to know that the spectral sequence degenerates at the $E_{n+2}$-page, which uses the results of \Cref{section5}, and also proves \Cref{theorembintro}.

\begin{proposition}\label{propD}
For $Z \in \GEM$, there are no differentials in the homotopy fixed point spectra sequence until the $E_{n+1}$ page. On the $E_{n+1}$ page, the differential is given by $d_{n+1}(\theta_n(a))= w_{n+1}\theta_n(da)$ for $a \in H^*(Z^{S^n})$.
\end{proposition}
\begin{proof}
Via the description of \Cref{omegapresentation} and \Cref{formalnonequivariantisomorphism}, the cohomology $H^*(Z^{S^n})$ is computed by the functor $L_n(H^*(Z))$. In particular, it is generated by classes of the form $\theta_n(a)$ and $\theta_n(da)$, where $a$ is a class in $H^*(Z)$ pulled back via evaluation at a point. By the Leibniz rule, it suffices to just understand differentials on $\theta_n(a)$, and $\theta_n(da)$ for these classes. But by naturality of the differential and $\theta$, and the fact that $a$ is pulled back from an Eilenberg--Mac Lane space, we just need to prove the assertions when $Z = K(\ZZ/2,m)$ and the cohomology of the fiber is a free unstable algebra on $a$ and $da$. By Lemma \ref{deltaimage}, the class $\theta_n(da)$ survives to $E_{\infty}$, so the unstable algebra generated by $\theta_n(da)$ cannot have any differentials by the universality argument along with the Leibniz rule. 

We first claim that $\theta_n(a)$ has no differentials before the $E_{n+1}$-page of the spectral sequence. To see this, we note that an earlier differential would be detected in $H^*(Z^{S^n}_{h\SOr(n)})$, so it suffices to show the homotopy fixed point spectral sequence computing $H^*(Z^{S^n}_{h\SOr(n)})$ degenerates. But the $\SOr(n)$-equivariant evaluation map $\mathrm{ev}_0:Z^{S^n} \to Z_0$ exhibits $a$ as a permanent cycle. This also shows that the $d_{n+1}$-differential cannot have terms involving $w_i$ for $i<n+1$.

To see the claimed $d_{n+1}$-differential, we consider the $\SOr(n+1)$-equivariant evaluation map $Z^{S^n}\times S^n\to Z_0$, and the induced map on homotopy fixed point spectral sequences. The homotopy fixed point spectral sequence for $Z_0$ degenerates since the action is trivial, so $a$ has no differentials. On the other hand, $a$ pulls back to $\theta_n(a)\otimes 1 + \theta_n(da) \otimes y_n$. In the homotopy fixed point spectral sequence for $S^n$, there is a $d_{n+1}$-differential $dy_n = w_{n+1}$ since the orbit is $\BSO(n)$. Thus the desired differential follows from the Leibniz rule since $0 = d_{n+1}(\theta_n(a)\otimes 1 + \theta_n(da) \otimes y_n) = d_{n+1}\theta_n(a)\otimes 1 + \theta_n(da) \otimes w_{n+1} = (d_{n+1}\theta_n(a) + \theta_n(da) w_{n+1} )\otimes 1$.
\end{proof}
Below is a picture of part of the homotopy fixed point spectral sequence, where $a \in H^m(Z)$.
\begin{center}
\begin{sseqdata}[ scale = 0.8, name = basic,
cohomological Serre grading ,xscale=1.5, yscale = 0.8, classes = {draw = none }, y axis gap = 28pt, class labels = { font = \small },y tick gap=0.7cm, right clip padding = 0.8cm,x tick handler = 
{
\ifnum#1 < 6\relax
    \ifnum#1 > 1
        #1
    \fi 
\else 
    \ifnum#1 = 6\relax 
       ...
    \else 
        n+1
    \fi 
\fi 
}, y tick handler = {
\ifnum#1 = 10
    2m-i
\fi 
\ifnum#1 = 2
    m-n 
\fi 
\ifnum#1 = 8
    m
\else 
\fi 
}]

\class["w_{2}"](2,0)
\class["w_{3}"](3,0)
\class["w_{4}"](4,0)
\class["w_{5}"](5,0)
\class["..."](6,0)
\class["w_{n+1}"](7,0)
\class["\theta_n(Sq_1 da)"](0,7)
\class["\theta_n(a)"](0,8)
\class["\theta_n(da)"](0,2)
\class["\theta_n(Sq_i a)"](0,10)
\class["\theta_n(da) \otimes w_{n+1}"](7,2)
\d7(0,8)
\end{sseqdata}
\printpage[ name = basic, page = 7 ] 

\end{center}

\begin{remark}\label{commute}
Note that the $E_{n+1}$ differential commutes with the Steenrod algebra action on the columns of the $E_2$-page of the spectral sequence. This is because it is given by the derivation $d$ multiplied by $w_{n+1}$, and $d$ commutes with the action of the Steenrod algebra by \Cref{omegapresentation}.
\end{remark}

We introduce some notation for dealing with excess of admissible sequences over the Steenrod algebra, using lower indices:

\begin{notation}\label{notation:admissible}
    Given a multi-index of nonnegative integers $(a_1,\dots,a_k)$, we say that the sequence is admissible of excess $<i$ in degree $m$ if the sequence of Steenrod operations $Sq_{a_1}\dots Sq_{a_k}$ acting on a class in degree $m$ is an admissible sequence of excess $<i$. This is equivalent to the condition that the $m-a_1 \leq i$, and the sequence is nondecreasing. Let $A_{b,c}$ be the set of admissible sequences of Steenrod operations of excess $<b$ in degree $c$.
\end{notation}

\begin{proposition}\label{generators of ss}
For $Z\in \GEM$, for $0 \leq i < n$ , the classes $\theta_n(da),\theta_n(Sq_ia), \theta_n(Sq_na+ada)$ for $a \in H^*(Z)$ along with the $w_i$ for $2 \leq i \leq n+1$  generate the $E_{n+2}$ page of the homotopy fixed point spectral sequence as an algebra.
\begin{proof}

Since the differential for the $E_{n+1}$ page is $d_{n+1}(\theta_n(a)) = w_{n+1}\theta_n(da)$, it suffices to show that $\ker(d) \subset H^*(Z^{S^n})$ is generated by $\theta_n(Sq_i a), \theta_n(Sq_n a + ada), \theta_n(da)$. Because all of the elements that are of the form $da$ are in the image, this is true if and only if the cohomology of $d$ acting on $H^*(Z^{S^n}) = L_n(H^*(Z))$ is generated by $Sq_ia$, $Sq_na+ada$ for $i <n$. 

Write $Z = \prod_{s \in S_m}K(\ZZ/2,m)$, where $S_*$ is a finite nonnegatively graded set. Then $L_n(H^*(Z))$ with its differential is a tensor product of $L_n(H^*(K(\ZZ/2,m)))$ over the set $S$. Thus we can reduce to the case when $Z = K(\ZZ/2,m)$. Denote $Sq_I = Sq_{a_1}Sq_{a_2}\ldots Sq_{a_k}$ for a multi-index $I = (a_1,\dots,a_k)$. Then $Sq_I \iota$ generate $H^*(Z)$ as a polynomial algebra when $I \in A_{m,m}$ (see \Cref{notation:admissible}).

$L_n(H^*(Z))$ is generated as an algebra by 
\begin{itemize}
    \item $Sq_I \iota$ with leading term $Sq_i$ for $i\neq n$ and $I \in A_{m,m}$
    \item $dSq_I\iota$ for $I \in A_{m-n,m}$
    \item $Sq_n Sq_I \iota + Sq_I \iota dSq_I \iota$ for $(n,I) \in A_{m,m}$
    
\end{itemize}
Indeed, these generators are obtained by adding decomposable elements to the usual set of generators. Next, observe that the generators listed above that are nonzero are actually a set of \textit{free generators} in the sense that the only relations among them are that $x^2 = x$ when $x$ is a generator of degree $0$.

$d(Sq_I \iota)$ is $0$ when the first term is $Sq_i$ with $i<n$, $d(Sq_I \iota)$ is $Sq_{I-n}d\iota$ where $I-n$ is the sequence where $n$ is subtracted from every term in $I$. Thus by pairing up each generator $x$ with $dx$ we have decomposed $H^*(X)$ as a tensor product of differential graded algebras. By the Kunneth formula, the cohomology of the differential graded algebra is generated by
\begin{enumerate}[label = (G\arabic*)]
    \item $Sq_0Sq_I \iota = (Sq_I \iota)^2$ where $I \in A_{m,m}$ has leading term $Sq_i$ for $i> n$.
    \item $Sq_I \iota$ with leading term $Sq_i$ for $i< n$.
    \item $Sq_n Sq_I \iota + Sq_I \iota dSq_I \iota$ for $(n,I) \in A_{m,m}$.
\end{enumerate}

We now analyze the cases $m\leq n$ explicitly for the reader.

In the case $m < n$, the differential $d$ on $L_n (H^*(Z))$ is zero, so the $E_{n+1}$ and $E_{n+2}$ pages are the same. On the other hand, $Sq_m a = a$ which along with the $w_i$ generate the $E_{n+1}$ page, so the claim holds.

In the case $m=n$, there is an `unusual' tensor factor which is the cdga $\FF_2[\iota,d\iota]/((d\iota)^2 = d\iota)$, and we can see that the cohomology is generated by $\iota^2$, $\iota (1+d\iota)$, which are in (G1) and (G3) respectively.
\end{proof}

\end{proposition}

We are now ready to prove \Cref{thm:mainreplacethis}, which we reformulate below:
\begin{theorem}[{\Cref{thm:mainreplacethis}}]\label{mainspecseq}
For $Z\in \GEM$, there are no differentials in the homotopy fixed point spectral sequence until the $E_{n+1}$ page. On the $E_{n+1}$ page, the differentials are determined by $d_{n+1}(\theta_n(a))= w_{n+1}\theta_n(da)$ and $d_{n+1}(\theta_n(da)) = 0$ for $a \in L_n(H^*(Z))$. Moreover, $E_{n+2} = E_{\infty}$.
\end{theorem}
\begin{proof}
This follows immediately from Proposition \ref{propF}, \ref{propD}, and \ref{generators of ss}.
\end{proof}

Next we prove \Cref{theorembintro}, which we reproduce below for convenience:
\begin{theorem}[{\Cref{theorembintro}}]\label{thmb}
	Let $Z$ be a space, and $n \geq 1$. For $0 \leq i \leq n$, there are natural transformations $$\delta: H^j(Z)\to H^{j-n}(Z^{S^n}_{h\SOr(n+1)})$$ $$\phi_i: H^j(Z) \to H^{2j-i}(Z^{S^n}_{h\SOr(n+1)})$$ such that if $Z \in \GEM$, the composite $H^*(Z^{S^n}_{h\SOr(n+1)}) \xrightarrow{q^*}
	H^*(Z^{S^n}) \cong L_n(H^*(Z))$ induced from the map $Z^{S^n} \to Z^{S^n}_{h\SOr(n+1)}$ sends  $\delta(a)$ to $da$, $\phi_i(a)$ to $Sq_ia$ for $i<n$, and $\phi_n(a)$ to $Sq_na + ada$.
	
	Moreover, when $Z \in \GEM$, the images of $\phi_i,\delta$ along with $w_i$ generate $H^*(Z^{S^n}_{h\SOr(n+1)})$ as an algebra.
\end{theorem}
\begin{proof}
Everything but the last statement is in Proposition \ref{propF}. For the last statement, by Proposition \ref{generators of ss} we learn that the image of the natural transformations generate the associated graded algebra, and hence also generate $H^*(Z^{S^n}_{h\SOr(n+1)})$ itself.
\end{proof}

\begin{remark}
\Cref{theorembintro} says that when $Z\in \GEM$, then all of the classes in $H^*(Z^{S^n}_{h\SOr(n+1)})$ can be constructed using $w_i, \delta, \phi_i$. The origin of $\delta,\phi_i$ are the equivariant evaluation maps $ev_0,ev_1$, so it's surprising that these classes account for the entire cohomology of $Z^{S^n}_{h\SOr(n+1)}$. Note that by \Cref{freemodule}, for any $\SOr(n+1)$-space, $H^*(X_{h\mathrm{O}(n)})$ is a free module over $H^*(X_{h\SOr(n+1)})$ and the pushforward is a surjective module map. Thus the crucial map in the construction of $\phi_i$ is not the pushforward, but rather the evaluation map.
\end{remark}

\section{Some relations between classes}\label{verificationsection}

The goal of this section is to make progress towards computing the cohomology of $Z^{S^n}_{h\SOr(n+1)}$ for $Z \in \GEM$, as an unstable algebra over the Steenrod algebra. From \Cref{thmb} we know that generators for $H^*(Z^{S^n}_{h\SOr(n+1)})$ are given by $w_i, \delta(a), \phi_i(a)$ for $a \in H^*(Z)$. Here we collect some relations that hold among these classes. The goal of this section is to prove the proposition below:
\begin{proposition}\label{extrarelations}
The following relations hold ($\delta_{i,j}$ is the Kronecker delta).\footnote{The numbering of the relations is so as to agree with the numbering in Definition \ref{lbardefn}.}

\begin{itemize}
    \item [(1)]$\phi_i(a+b) = \phi_i(a) + \phi_i(b)+w_{n-i}\delta(ab)$ ($w_1=0,w_0=1$)
    \item [(2)]$\delta(a+b) = \delta(a)+\delta(b)$
\item [(4)] $w_{n+1}\delta(a) = 0$
    \item [(6)]$\phi_k(ab) = \sum_{i+j=k}\phi_i(a)\phi_j(b) + \sum_{\ell = n+1}^{2n} \sum_{i+j=\ell} \phi_i(a)\phi_j(b)\left( \sum_{\substack{2 \leq \alpha_1 \ldots \alpha_m \leq n+1\\ \alpha_1 + \ldots \alpha_m = \ell-k\\ \alpha_m>n-k} }\prod_{f=1}^m w_{\alpha_f}\right)$
    \item [(7)]    $Sq\phi_{n-k}(a)=\sum_{j\geq0} \sum_{i}\binom{k+|a|-j}{i-2j}\phi_{n-i-k+2j}(Sq^i a)+\delta_{k0}\sum_{2j<i}\delta(Sq^j a\times Sq^{i-j} a)$
    \item [(8)] $Sq^i\delta(a)=\delta(Sq^i a)$
    \item [(9)]$Sq_i(\delta(a)) = Sq_i({\phi}_j(a)) = 0$ for $i <0$\footnote{This relation can be ignored because it is the instability condition.}
    \item [(10)]$Sq_0({\delta}(a))= {\delta}(a)^2$
\end{itemize}

\end{proposition}

To begin verifying the relations, we study the following diagram, where $\SOr(n)$ acts trivially on the spaces in the left side of the diagram:
\begin{center}
\begin{equation}\tag{i}\label{deltadiagram}
\begin{tikzcd}
H^*(Z_{h\SOr(n)}) \arrow[r, "ev_0^*"]                                                           & H^*(Z^{S^n}_{h\SOr(n)}) \arrow[r, "\tau_{\SOr(n)}^{\SOr(n+1)}"]\ar[d,equal] & H^*(Z^{S^n}_{h\SOr(n+1)}) \\
H^*(N_{\SOr(n)}^{\Or(n)} (Z)_{h\SOr(n)}) \arrow[d, "\tau_{\SOr(n)}^{\mathrm{O}(n)}"'] \arrow[r, "ev_1^*"] \arrow[u, "\Delta^*"] & H^*(Z^{S^n}_{h\SOr(n)}) \arrow[d, "\tau_{\SOr(n)}^{\mathrm{O}(n)}"']   &                         \\
H^*(N_{\SOr(n)}^{\Or(n)} (Z)_{h\mathrm{O}(n)}) \arrow[r, "ev_1^*"]                                                          & H^*(Z^{S^n}_{h\mathrm{O}(n)}) \arrow[ruu, "\tau_{\mathrm{O}(n)}^{\SOr(n+1)}"'] &                        
\end{tikzcd}
\end{equation}
\end{center}

The only square that does not obviously commute is the top left one. But the corresponding maps of spaces commutes up to homotopy since evaluation at any two points give homotopic maps as $S^n$ is connected.

\begin{lemma}\label{relation124}
The class $a\otimes b \in H^*(Z^2) $ maps to $\sum_0^nt^i r^*(w_i\delta(ab))$ in $H^*(Z^{S^n}_{h\mathrm{O}(n)})$ in (\ref{deltadiagram}), where $r: Z^{S^n}_{h\mathrm{O}(n)} \to Z^{S^n}_{h\SOr(n+1)}$ is the quotient map, $r^*(w_1)=0, r^*(w_0)=1$. This implies $w_{n+1}\delta(a)=0$. 
\begin{proof}
Moving the class $a\otimes b$ through $\Delta^*$ and the top row of the diagram gives $\delta(ab)$. $H^*(Z^{S^n}_{h\mathrm{O}(n)})$ is a free module over $H^*(Z^{S^n}_{h\SOr(n+1)})$ with basis $t^i$, $0 \leq i \leq n$ (see \Cref{freemodule}), and $\tau_{\mathrm{O}(n)}^{\SOr(n+1)}$ extracts the coefficient of $t^n$. Thus the $t^n$ coefficient of the image of $a\otimes b$ is $\delta(ab)$. Now in $H^*(N_{\SOr(n)}^{\Or(n)} (Z)_{h\mathrm{O}(n)})$, $\tau_{\SOr(n)}^{\mathrm{O}(n)}a\otimes b$ is killed by multiplication by $t$, so the same is true of the image of $a\otimes b$. But if $q$ is the map $Z^{S^n}_{h\mathrm{O}(n)} \to Z^{S^n}_{h\SOr(n+1)}$, then $t^{n+1}=t^{n-1}q^*w_2+t^{n-2}q^*w_3 + \dots+ q^*w_{n+1}$. Thus if $t(t^n\delta(ab)+\sum_0^{n-1}c_it^i)=0$, we must have that $c_i=\delta(ab)q^*w_{n-i}$ ($c_{n-1}=0$), and $q^*w_{n+1}\delta(ab)=0$, giving the result.

\end{proof}
\end{lemma}

Next, we look at the Steenrod action on $\delta(a),\phi_i(a)$. There is only one step in the construction of $\delta(a),\phi_i(a)$ that does not necessarily commute with the Steenrod action, namely the pushforward map. Nevertheless, the interaction of the pushforward with the Steenrod squares is understandable via the Stiefel--Whitney classes of the normal bundle of the fibration.

\begin{lemma}\label{pushforwardsquares}
Let $f:X \to Y$ be a fibration whose fibers are manifolds of dimension $n$, let $\nu$ denote the normal bundle of $f$, $w_\nu$ its total Stiefel-Whitney class, and $\tau_f$ the pushforward map of $f$. Then $Sq(\tau_f(a)) = w_\nu Sq(a)$.
\begin{proof}
Recall that the pushforward is the composite of the Thom isomorphism of $\nu$ with the the Gysin map $H^*(X^{\nu}) \to H^*(Y)$ in cohomology. Since the latter map comes from a map of spectra, it commutes with Steenrod operations, and the former is given by multiplication by the Thom class $u$. $Squ=w_\nu u$ essentially by definition of the $w_i$ \cite{milnorstasheff}, so $Sq(ua) = SquSqa=w_\nu uSqa$. Then composing with the Gysin map gives the result.
\end{proof}
\end{lemma}
\begin{lemma}
\label{relation78}
The relations
    \[Sq\phi_{n-k}(a)=\sum_{j\geq0} \sum_{i}\binom{k+|a|-j}{i-2j}\phi_{n-i-k+2j}(Sq^i a)+\delta_{k0}\sum_{2j<i}\delta(Sq^j a\times Sq^{i-j} a)\]
    \[Sq^i\delta(a)=\delta(Sq^i a)\] hold.

\begin{proof}
For $\delta(a)$, by Lemma \ref{pushforwardsquares}, we have that $\Sq\delta(a) = w_\nu\delta(\Sq(a))$. However, $w_\nu$ is easy to compute: $\nu$ is the pullback of the normal bundle of the fibration $\BSOr(n) \to \BSOr(n+1)$. But $-\nu$ is the bundle that gives the tangent of the fiber, which is the tautological bundle with Stiefel--Whitney classes $w_i$. Thus $w_\nu = (\sum w_i)^{-1}$.

For $\phi_i(a)$, because we started with the class $t^iP(a)$, it suffices to understand the action of the Steenrod squares on that. This is given by the Nishida relations (see \cite{nishidarelations}):
$$\Sq(t^kP(a)) = \sum_{i>0}\left(\sum_{j\geq0} \binom{k+|a|-j}{i-2j}t^{i+k-2j}P(\Sq^j a) + \delta_{k0}\sum_{2j<i}\tau_{*}^{C_2}(Sq^j a\otimes Sq^{i-j} a)\right)$$
which gives the desired result. 
\end{proof}
\end{lemma}

We are now ready to prove \Cref{extrarelations}.
\begin{proof}
Relations $(9),(10)$ hold for any object of $\cK$. Relation $(4)$ was proven in Lemma \ref{relation124} and Relations $(7),(8)$ were proven in Lemma \ref{relation78}.

We now prove the remaining relations, using the notation of Section \ref{section5}. First observe that in $H^*(N_{\SOr(n)}^{\Or(n)} (Z)_{h\mathrm{O}(n)})$, $\tau_{\SOr(n)}^{\mathrm{O}(n)}(a\otimes b) = P(a+b)-P(a)-P(b)$. Thus multiplying by $t^i$ and looking at the image in $H^*(Z^{S^n}_{h\SOr(n+1)})$, we can use the definition of the $\phi_i$ to get relation $(1)$. Since $ev^*_0$ and $\tau_{\SOr(n)}^{\SOr(n+1)}$ are additive, relation $(2)$ holds. To see relation $(3)$, note that the image of $P(a)$ in $Z^{S^n}_{h\mathrm{O}(n)}$ is $\sum_i^n t^i \phi_i(a)$. Since $P(ab) = P(a)P(b)$, and $t^{n+1}=t^{n-1}q^*w_2+t^{n-2}q^*w_3 + \dots+ q^*w_{n+1}$, isolating the $t^k$ coefficient gives 
$$\phi_k(ab) = \sum_{i+j=k}\phi_i(a)\phi_j(b) + \sum_{\ell = n+1}^{2n} \sum_{i+j=\ell} \phi_i(a)\phi_j(b)\left( \sum_{\substack{2 \leq \alpha_1 \ldots \alpha_m \leq n+1\\ \alpha_1 + \ldots \alpha_m = \ell-k\\ \alpha_m>n-k} }\prod_{f=1}^m w_{\alpha_f}\right).$$
\end{proof}

\section{A presentation of the zeroth derived functor}\label{algebraicstuffsection}

Having computed $H^*(Z^{S^n}_{h\SOr(n+1)})$ for $Z\in \GEM$ in terms of an explicit cdga, we turn to giving an explicit presentation of the cohomology of that cdga. We do this by writing an explicit functor $\overline{\ell}_n$, and proving that $\overline{\ell}_n$ is the zeroth derived functor of the associated graded algebra of $H^*(Z^{S^n}_{h\SOr(n+1)})$. Our approach is largely inspired by and analogous to \cite[Section 5, Section 8]{Bkstedt1999HomotopyOO}.

Our first goal will be to define the functor $\overline{\ell}_n$. Because of the fact that the differentials in \Cref{thm:mainreplacethis} are compatible with the Steenrod operations acting on the columns of the spectral sequence, $\overline{\ell}_n$ naturally lands in a slightly more refined category than bigraded commutative $\FF_2$-algebras. We now define this more refined target category, $\cK^{\gr }$.

\begin{definition}
    We let $\cK^{\gr }$ denote the category of commutative algebras in the category $\cU^{\NN}$ (with respect to the Day convolution\footnote{This means that $(X\otimes Y)_i= \oplus_{j+k=i}X_j\otimes Y_k$}), such that the degree $0$ piece is in $\cK$. We refer to the degree coming from the grading of objects in $\cK$ and $\cU$ as the \textit{Steenrod degree}, to distinguish it from degree coming from the grading.
\end{definition}

We now define the functor whose underlying bigraded algebra we claim to be the zeroth derived functor of $\gr H^*(Z^{S^n}_{h\SOr(n+1)})$.
\begin{definition}\label{lbardefn}
Let $A \in \cK$. $\overline{\ell}_n(A)$ is the graded algebra over the Steenrod algebra multiplicatively generated as an object in $\cK^{\gr }$ by classes $\overline{\phi}_i(a),\overline{\delta}(a)$ for $a \in A$ in graded degree $0$, and $\overline{w_i}$ in graded degree $i$, with Steenrod degree $2|a|-i, |a|-n, $and $0$ respectively, with the following relations, where $\delta_{i,j}$ denotes the Kronecker delta function:
\begin{enumerate}
\item $\overline{\phi}_i(a+b) = \overline{\phi}_i(a) + \overline{\phi}_i(b)$
    \item $\overline{\delta}(a+b) = \overline{\delta}(a)+\overline{\delta}(b)$
    \item $\overline{\delta}(ab)\overline{\delta}(c)+\overline{\delta}(bc)\overline{\delta}(a)+\overline{\delta}(ca)\overline{\delta}(b)=0$
    \item $\overline{w}_{n+1}\overline{\delta}(a)=0$
    \item $\overline{\delta}(a)\overline{\phi}_i(b) = \overline{\delta}(aSq_ib)+{\delta}_{0n}\overline{\delta}(ab)\overline{\delta}(b)$
    \item $\overline{\phi}_k(ab) = \sum_{i+j=k}\overline{\phi}_i(a)\overline{\phi}_j(b)$
    \item $Sq\overline{\phi}_{n-k}(a)=\sum_{j\geq0} \sum_{i}\binom{k+|a|-j}{i-2j}\overline{\phi}_{n-i-k+2j}(Sq^i a)+{\delta}_{k0}\sum_{2j<i}\overline{\delta}(Sq^j a\times Sq^{i-j} a)$
    \item $Sq^i\overline{\delta}(a)=\overline{\delta}(Sq^i a)$
    \item $Sq_i(\overline{\delta}(a)) = Sq_i(\overline{\phi}_j(a)) = 0$ for $i <0$
    \item $Sq_0(\overline{\delta}(a))= \overline{\delta}(a)^2$, $Sq_0(\overline{\phi}_j(a)) = \overline{\phi}_j(a)^2$
    \item $\overline{\phi}_k(Sq_i a)  = \sum_{j=0}^{(2|a|-i-k)/2}\binom{|a|-i-j-1}{2|a|-i-k-2j}\overline{\phi}_{2j-2|a|+2i+k}(Sq^j a) + {\delta}_{k,n}\overline{\phi}_{i}(a)\overline{\delta}(Sq_i a)$ for $k>i$
    \item $\overline{\phi}_0(Sq_j a) = \overline{\phi}_j(a)^2 +{\delta}_{j,n}\overline{\delta}(a^2 Sq_n a)$
\end{enumerate}
where we declare $\overline{\phi}_m(a) = 0$ for $m>n$.
\end{definition}

Now, we can construct a natural transformation $\overline{\eta}_n$ on GEMs that we will show witnesses $\overline{\ell}_n$ as the zeroth derived functor of $H^*(Z^{S^n}_{h\SOr(n+1)})$.

\begin{proposition}\label{constructeta}
Let $Z\in h\cS^{\ft}$. For $a \in H^*(Z)$, the map sending $\overline{w}_i,\overline{\delta}(a),\overline{\phi}_i(a)$ to the image of $w_i,\delta(a),\phi_i(a)$ in the associated graded, with $w_i$ in grading $i$ and $\delta(a), \phi_i(a)$ in grading $0$, defines a natural map of bigraded $\FF_2$-algebras $$\overline{\eta}_n:\overline{\ell}_n(H^*(Z)) \to \gr H^*(Z^{S^n}_{h\SOr(n+1)})$$
This map is an isomorphism for $Z \in \GEM^{\fin}$ iff it witnesses $\overline{\ell}_n$ as the zeroth derived functor of $\gr H^*(Z^{S^n}_{h\SOr(n+1)})$.
\end{proposition}
\begin{proof}
$\overline{\ell}_n$ preserves sifted colimits since it is defined from the input in terms of generators and relations. Sifted colimits are preserved by the forgetful functor from $\cK^{\gr}$ to bigraded $\FF_2$-algebras, so if the formula in the proposition defines a natural transformation for $Z \in \GEM^{\fin}$, then it left Kan extends to a map to the zeroth derived functor of $\gr H^*(Z^{S^n}_{h\SOr(n+1)})$, which naturally maps to $\gr H^*(Z^{S^n}_{h\SOr(n+1)})$ by \Cref{lemma:nattransfinitetype}, giving the map in general. We also see that $\overline{\eta}_n$ being an isomorphism on $Z \in \GEM^{\fin}$ is equivalent to it witnessing $\overline{\ell}_n$ as the zeroth derived functor of $\gr H^*(Z^{S^n}_{h\SOr(n+1)})$.

Thus to finish the proof of the proposition, because $\overline{\ell}_n$ is presented by generators and relations, it suffices to check that all of the generating relations among $\overline{w}_i,\overline{\delta}(a),\overline{\phi}_i(a)$ are satisfied in the declared images for $Z \in \GEM^{\fin}$. We check the relations in the same order that the relations are listed in \Cref{lbardefn}.
\begin{enumerate}
    \item This is linearity of $Sq_i$.
    \item This is linearity of $d$.
    \item This follows from the Leibniz rule for $d$.
    \item This follows since the $d_{n+1}$ differential applied to $a$ is $w_{n+1}da$.
    \item We can calculate $d(a)Sq_i(b) = d(aSq_ib)$ for $i<n$ and $d(a)(Sq_n(b)-bdb) = d(ab)db$.
    \item This is essentially the Cartan formula.
    \item Lemma \ref{relation78}.
    \item Lemma \ref{relation78}, or alternatively the fact that $Sq_i$ commutes with $d$.
    \item This is the instability condition.
    \item This is the compatibility of $Sq_0$ with the algebra structure. 
    \item This is the Adem relations.
    \item This follows from the fact that $Sq_0(Sq_ia) = (Sq_ia)^2$.
\end{enumerate}
\end{proof}

The goal of this section is to prove \Cref{thmc} below, which by \Cref{constructeta} (and Postnikov completeness) is equivalent to checking $\overline{\eta}_n$ is an isomorphism for $Z \in \GEM^{\fin}$. Note that this implies \Cref{thmcintro} of the introduction.

\begin{theorem}\label{thmc}
The natural transformation $\overline{\eta}_n$ realizes $\overline{\ell}_n$ as the zeroth derived functor of $Z \mapsto grH^*(Z^{S^n}_{h\SOr(n+1)})$, and is an isomorphism for $Z \in \GEM$.
\end{theorem}

The work done here is completely algebraic. To make the identification, we break up both sides into manageable pieces, and identify the pieces.

First we define functors $\overline{E}_{n+1},\overline{E}_{\infty}$ that describe algebraically the $E_{n+1}$-page and $E_{\infty}$-page of the homotopy fixed point spectral sequence for $Z \in \GEM$.

\begin{definition} Given $A \in \cK$, $\overline{E}_{n+1}(A)$ is defined as the bigraded algebra given by $L_n(A)\otimes H^*(\BSOr(n+1))$ where the first tensor factor is in bidegrees $(0,*)$ and the second is in bidegree $(*,0)$. The Steenrod square $Sq^i$ acts in the $(0,*)$ direction, so that $\overline{E}_{n+1}$ can be viewed as an object of $\cK^{gr}$. $\overline{E}_{n+1}(A)$ is additionally equipped with a differential $d$ degree $(-n,n+1)$ that is a derivation, determined by the formula $d(a) = da\otimes w_{n+1}$ and $d(da) = 0$ , where $a, da$ are the generators of $L_n(A)$.
\end{definition}

\begin{definition}
The functor $\overline{E}_{\infty}:\cK \to \cK^{gr}$ defined to be the cohomology of $\overline{E}_{n+1}(A)$ with respect to the differential $d$.
\end{definition}

$\overline{E}_{\infty}$ of course depends on $n$, but we suppress the dependence in the notation.

\begin{lemma}
When $Z \in \GEM$, $\overline{E}_{\infty}$ can be identified with the $E_{\infty}$ term of the homotopy fixed point spectral sequence.
\begin{proof}
Recall that for $Z \in \GEM$, the $E_{n+1}$ page of the spectral sequence is the tensor product of $H^*(\BSOr(n+1))$ and $H^*(Z^{S^n})=L_n(H^*(Z))$ by Proposition \ref{formalnonequivariantisomorphism}, and by Theorem \ref{mainspecseq} the result immediately follows.
\end{proof}
\end{lemma}

\begin{definition}\label{secondoverlineeta}
Define the natural transformation $\overline{\eta}_n:\overline{\ell}_n \to \overline{E}_{\infty}$ which sends $\overline{\delta}(a)\mapsto da$, $\overline{\phi}_i(a) \mapsto Sq_ia$ for $i<n$, and $\overline{\phi}_n(a) \mapsto Sq_na + ada$. These images in $L_n$ clearly are in the kernel of $d$ and hence in $\overline{E}_{\infty}$, and all the relations of $\overline{\ell}_n$ are satisfied.
\end{definition}

\begin{remark}
Definition \ref{secondoverlineeta} is compatible with our previous definition of $\overline{\eta}_n$ when $A=H^*(Z)$ for $Z \in \GEM$.
\end{remark}

In order to study $\overline{\ell}_n$, we break it up into smaller pieces which contain essentially all of its information.

\begin{definition}
$\ell_n'$ is defined to be the quotient of $\overline{\ell}_n$ as a ring by $\overline{w}_i,i\leq n+1$. $\ell_n'/\delta$ is defined to be the quotient of $\ell_n'$ by $\overline{\delta}(a)$ for all $a$.
\end{definition}

Note that $\ell_n'$ and $\ell_n'/\delta$ live entirely in grading zero, so can be viewed as objects in $\cK$.

\begin{lemma}
$\ell_n'$ can be presented in $\cK$ using the generators $\delta(a),\phi_i(a)$, and the relations (1)-(3) and (5)-(12) of \Cref{lbardefn}. Furthermore, $\ell_n'/\delta(A)$ is generated by $\overline{\phi}_i(a)$ for $a \in A$ modulo the following relations:
 
\begin{enumerate}
\item $\overline{\phi}_i(a+b) = \overline{\phi}_i(a) + \overline{\phi}_i(b)$
    \item $\overline{\phi}_k(ab) = \sum_{i+j=k}\overline{\phi}_i(a)\overline{\phi}_j(b)$
    \item $Sq^i\overline{\phi}_{n-k}(a)= \sum_{j}\binom{k+|a|-j}{i-2j}\overline{\phi}_{n-i-k+2j}(Sq^j a)$
    \item $ Sq_i(\overline{\phi}_j(a)) = 0$ for $i <0$
    \item $Sq_0(\overline{\phi}_j(a)) = \overline{\phi}_j(a)^2$.
    \item $\overline{\phi}_k(Sq_i a)  = \sum_{j=0}^{(2|a|-i-k)/2}\binom{|a|-i-j-1}{2|a|-i-k-2j}\overline{\phi}_{2j-2|a|+2i+k}(Sq^j a)$ for $k>i$.
    \item $\overline{\phi}_0(Sq_j a) = \overline{\phi}_j(a)^2$
\end{enumerate}
with $\overline{\phi}_m(a) = 0$ for $m>n$.
\begin{proof}
The first statement is clear since the $\overline{w}_i$, $i\leq n$ are not involved in the relations, and the only relation on $\overline{w}_{n+1}$ is $\overline{w}_{n+1}\overline{\delta}(a)=0$.

The second statement is also clear since $\ell_n'/\overline{\delta}(A)$ is obtained from $\ell_n'(A)$ by adding more relations, namely killing the $\overline{\delta}$ classes.
\end{proof}
\end{lemma}

We break up the functor $\overline{E}_{\infty}$ into smaller pieces.

\begin{definition}
${E}'_{\infty}(A)$ is the quotient of $\overline{E}_{\infty}(A)$ by the $w_i$. ${E}'_{\infty}/d(A)$ is the quotient of ${E}'_{\infty}(A)$ by elements of the form $da$.
\end{definition}

The natural map $\overline{\ell}_n \to \overline{E}_{\infty} \to {E}'_{\infty}$ sends the $w_i$ to $0$, so factors as a map $\eta_n':\ell_n' \to {E}'_{\infty}$. Further quotienting the map to $E'_{\infty}/d$ factors makes the map factor through $\ell_n'/\delta$, giving a map we can call $\eta_n'/d$.

This implies $E'_{\infty}(A),E'_{\infty}(A)/d$ can also be described as follows: Consider $L_n(A)$ as a differential graded algebra, where the differential is $a \mapsto da$. Then $E'_{\infty}(A)$ is the kernel of $d$ and $E'_{\infty}(A)/\delta$ is the cohomology of $d$.

Finally, we need an algebraic version of the pushforward map $\tau_{\SOr(n)}^{\SOr(n+1)}$, and a ring map from $\ell_n'(A)$ into a subalgebra of $L_n(A)$ that is an algebraic version of the quotient map in cohomology.

\begin{definition}
The map $\tau:L_n(A) \to \ell_n'(A)$ is the $\FF_2$-linear map that sends $adb_1\dots db_n$ to $\delta (a)\delta (b_1) \dots \delta (b_n)$. The map $i:\ell_n'(A) \to L_n(A)$ is the ring map sending $\delta(a)$ to $da$, $\phi_i(a)$ to $Sq_ia$ for $i < n$ and $\phi_n(a)$ to $Sq_na + ada$.
\end{definition}

\begin{proposition}\label{algebraicpushforward}
The maps $i,\tau$ are well-defined, and the differential $d$ on $L_n$ factors as $i \circ \eta_n'\circ \tau$. Moreover, $\tau \circ \eta_n' = 0$ and $\tau$ fits in an exact sequence $L_n(A) \xrightarrow{\tau} \ell_n'(A) \to \ell_n'/\delta(A) \to 0$.

\begin{proof}
To check $\tau$ is well defined, it suffices to check the ideal generated by relations $(1),(2),(3),(4)$ from Proposition \ref{omegapresentation} is sent to $0$. Relation $(1)$ is easy and relation $(3),(4)$ follows from the fact that $\ell_n' \in \cK$. To check $(2)$, if we have an element of the form $adb_1\dots db_n(d(xy)-d(x)y+d(y)x)$, it is sent to $\delta(b_1)\dots\delta(b_n)(\delta(ay)\delta(x)+\delta(xy)\delta(a)+\delta(ax)\delta(y)) = 0$.

To check $i$ is well defined, we must check that the relations for $\ell_n'$ hold in $L_n(A)$. To see this, by naturality, it suffices to prove the relations hold in the universal cases. But the universal cases are GEM spaces, where we have an identification of $L_n(H^*(Z))$ with $H^*(Z^{S^n})$ which agrees with our defined map.

\end{proof}
\end{proposition}

\begin{lemma}\label{isomorphismreduction}
The natural map $\overline{\eta}_n:\overline{\ell}_n(A) \to \overline{E}_{\infty}(A)$ is an isomorphism iff $\eta_n'/d(A)$ is an isomorphism.
\begin{proof}
We observe from the presentation that $\overline{\ell}_n(A)$ decomposes as a direct sum 
$$\bigoplus_{\alpha_i \geq 0} w_2^{\alpha_2} \cdots w_{n}^{\alpha_{n}} \ell_n'(A)\oplus \bigoplus_{\beta_i \geq 0, \beta_{n+1}>0} w_2^{\beta_2} \ldots w_{n+1}^{\beta_{n+1}} \ell_n'/\delta(A)$$

Similarly, $\overline{E}_{\infty}$ decomposes as
$$\bigoplus_{\alpha_i \geq 0} w_2^{\alpha_2} \cdots w_{n}^{\alpha_{n}} E_{\infty}'\oplus \bigoplus_{\beta_i \geq 0, \beta_{n+1}>0} w_2^{\beta_2} \ldots w_{n+1}^{\beta_{n+1}} E_{\infty}'/d$$

Via this decomposition, the map $\overline{\ell}_n(A) \to \overline{E}_{\infty}(A)$ splits as a sum of copies of $\eta_n'/d(A)$ and $\eta_n'(A)$ accordingly. Thus it suffices to prove that $\eta_n'/d(A)$ and $\eta_n'(A)$ are isomorphisms.

We now show that $\eta_n'/d$ being an isomorphism implies $\eta_n'$ is too. $E_{\infty}'$ is generated by lifts of $E_{\infty}'/d$ as well as the elements $da$. Because $\eta_n'/d$ is surjective, and $da$ is hit by $\delta(a)$, the map is surjective. For injectivity, suppose $a$ is in the kernel of $\eta_n'$. Then because $\eta_n'/d$ is injective, we can assume $a$ is in the kernel of $d$. But by Proposition \ref{algebraicpushforward}, this means $a = \tau(b)$ where $db=0$. By surjectivity of $\eta_n'/d$ this means that $b$ is the sum of $dc$ and a class in the image of $\eta_n'$. But $\tau(dc) = 0$ and $\tau \circ \eta_n' = 0$, which shows $a = 0$.
\end{proof}
\end{lemma}

The next lemma is the key to proving that $\overline{\eta}_n$ is an isomorphism.

\begin{lemma}\label{Eilenbergmaclanecomputation}
For $Z \in \GEM$, $\eta_n'/d$ is an isomorphism for $A=H^*(Z)$.
\end{lemma}

\begin{proof}
We will use notation from Proposition \ref{generators of ss} (see \Cref{notation:admissible}).
Write $Z = \prod_{s \in S_m}K(\ZZ/2,m)$, where $S_*$ is a finite nonnegatively graded set, set $|s| = m$ for an element $s \in S_m$, and let $A_{|s|,|s|}$ be the admissible sequences of excess $<|s|$ of degree $<|s|$ (see \cref{notation:admissible}).
In the proof of Proposition \ref{generators of ss}, it is shown that $E'_{\infty}/d(H^*(Z))$ is generated by the nonzero elements of the below list:
\begin{enumerate}[label = (G\arabic*)]
    \item $Sq_0Sq_I \iota_s$ for $(i_1,\dots,i_k)=I \in A_{|s|,|s|}$ satisfying $i_1> n$.
    \item $Sq_I \iota_s$ with leading term $Sq_i$ for $i< n$.
    \item $Sq_n Sq_I \iota_s + Sq_I \iota_s dSq_I \iota_s$ for $(n,I) \in A_{|s|,|s|}$
\end{enumerate}
Moreover, the only relation among the generators is when  $|s|=0$, where the relation $\iota_s^2 = \iota_s$ holds.

We can thus define an algebra map $\zeta: E_\infty' (H^*(Z))/d \to \ell_n'(H^*(Z))/d$ via the following prescription:

\begin{itemize}
    \item Send $Sq_0Sq_I \iota_s$ in $(G1)$ to $\phi_0(Sq_I\iota_s)$.
    \item Send $Sq_I\iota_s$ in $(G2)$ to $\phi_i(Sq_{I'}\iota_s)$, where $I'$ is $I$ with $i$ removed.
    \item Send $Sq_n Sq_I \iota_s + Sq_I \iota_s dSq_I \iota_s$ in $(G3)$ to $\phi_n(Sq_I\iota_s)$.
\end{itemize}

Relation $(7)$ for $\ell_n'/\delta$ shows that this does give an algebra homomorphism, and it is clear that that $\eta_n'/d \circ \zeta = id$, implying that $\zeta$ is injective. Thus it remains to check that $\zeta$ is surjective. Due to the multiplicative and additive relations for $\phi_i$, it suffices to check that $\phi_i(g)$ is in the image for each multiplicative generator $g$ of $H^*(Z)$.
\begin{enumerate}
    \item $\phi_0(Sq_I\iota_s)$ is automatically in the image for $Sq_I$ leading with $Sq_i$ for $i >n$. For $i\leq n$, relation $(7)$: $(\phi_i(a))^2 =\phi_0(Sq_i a)$ shows that it is in the image. 

\item For $i>0$, $\phi_i(Sq_I \iota)$ is automatically in the image for $Sq_I$ leading with $Sq_j$ for $j \geq i$. For $j<i$, the Adem relations $(6)$ lets us express $\phi_i(Sq_I \iota)$ using terms involving 
$$\sum \phi_\alpha(Sq_\beta \iota)$$
with $\alpha \leq \beta < j$ which is in the image.
\end{enumerate}
\end{proof}

Now we can complete the proof of Theorem \ref{thmc}.

\begin{proof}[Proof of \Cref{thmc}]
By \Cref{constructeta}, it suffices to see that $\overline{\eta}_n$ is an isomorphism for $Z \in \GEM$, but this follows from Lemma \ref{isomorphismreduction} and Lemma \ref{Eilenbergmaclanecomputation}.
\end{proof}

\printbibliography 
\end{document}